\documentclass[reqno,11pt]{amsart}
\usepackage{amssymb,amsmath,mathrsfs}
\usepackage{enumerate,tikz,bbm}
\usepackage{enumitem}
\usepackage{hyperref,fontawesome}

\usetikzlibrary{shapes.arrows,bending}
\usetikzlibrary{decorations.pathreplacing}
\usetikzlibrary{decorations.markings}
\usetikzlibrary{snakes,patterns}

\usetikzlibrary{shapes.geometric}
\usetikzlibrary{shapes.misc}

\tikzstyle{poly}=[outer sep=0pt, draw, line width=1.5, gray!70]
\tikzstyle{poly2}=[outer sep=0pt, draw, line width=1.3, brown!70]
\tikzstyle{root}=[draw, line width=1.5, gray!70,-{latex[blue!60,bend]}]
\tikzstyle{v}=[circle, inner sep=1.5, fill]
\tikzstyle{v0}=[circle, inner sep=1.8, fill=white, draw=black, thick]
\tikzstyle{v1}=[circle, inner sep=1.8, fill=orange!80!black, draw=black, thick]
\tikzstyle{road}=[rounded corners, line width=12, line cap=round, opacity=0.12]
\tikzstyle{scribble}=[snake=coil, segment aspect=0, line width=5, line cap=round, opacity=0.12]
\tikzstyle{label}=[brown!70]

\numberwithin{equation}{section}
\newtheorem{theorem}[equation]{Theorem}
\newtheorem{lemma}[equation]{Lemma}
\newtheorem{proposition}[equation]{Proposition}
\newtheorem{corollary}[equation]{Corollary}

\theoremstyle{definition}

\newtheorem*{remark}{Remark}

\def\merge{\mathbin{+\mkern-10mu+}}

\def\U{\mathsf{U}}
\def\D{\mathsf{D}}

\newcommand\dyckpath[3]{
\def\diam{0.08}
  \begin{scope}
    \draw[help lines] (#1) -- ++(#2*2,0);
    \draw[line width=1pt] (#1) foreach \dir in {#3}{ -- ++(\dir*90-45:1.41)};
    \draw[fill] (#1) circle (\diam);
    \draw[fill] (#1) foreach \dir in {#3}{ ++(\dir*90-45:1.41) circle (\diam)};
   \end{scope}
}

\definecolor{red1}{HTML}{B02400}
\colorlet{insertcol}{cyan!50}

\title{Rooted bicubic planar maps via Dyck paths}
\author{Juan B. Gil}
\author{Jackie N. Kaminski}
\email{jgil@psu.edu}
\email{jnk15@psu.edu}
\address{Penn State Altoona\\ 3000 Ivyside Park\\ Altoona, PA 16601}


\begin{document}

\begin{abstract}
We provide a combinatorial proof of Tutte's decomposition of rooted bicubic planar maps into 3-connected components. Motivated by the framework of Bell transformations, we establish an explicit bijection between rooted bicubic planar maps on $2n$ vertices and Dyck paths of semilength $3n$ with ascents of length divisible by 3, where each $3j$-ascent is colored using one of $g_j$ colors corresponding to the rooted 3-connected bicubic maps on $2j$ vertices. Our bijection gives a constructive method for assembling all rooted bicubic planar maps from their 3-connected building blocks. We give a simple proof for the fact that every 3-connected bicubic planar map on $2n$ vertices with $n \geq 4$ can be obtained from a smaller primitive map through just two insertion operations that add either 4 or 6 vertices. Finally, we briefly discuss rootings of 3-connected bicubic maps, providing lower bounds on the minimal number of rootings and showing that prism graphs can be used in combination with our insertion operations to generate maps with the maximum of $6n$ distinct rootings for all $n \geq 11$.
\end{abstract}

\maketitle

\section{Introduction}

In 1963, Tutte \cite[Section~11]{Tutte} observed that ``Each rooted bicubic map can be represented as a multiple extension of a 3-connected bicubic core,'' a structural fact he encoded in the functional equation
\begin{equation} \label{eq:cubicMaps}
 F(x) = G\big(x(1+F(x))^3\big)
\end{equation}
relating the generating function $F(x)$ of rooted bicubic maps on $2n$ vertices to the generating function $G(x)$ of their 3-connected counterparts.

Let $F(x)=\sum\limits_{n=1}^\infty f_n x^n$ and $G(x)=\sum\limits_{n=1}^\infty g_n x^n$. As shown by Birmajer, Gil, and Weiner~ \cite{BGW19}, equation \eqref{eq:cubicMaps} means that the sequence $f=(f_n)_{n\in\mathbb{N}}$ is the Bell transform of $g=(g_n)_{n\in\mathbb{N}}$ with parameters $(3,0,-1,1)$. Consequently, we get
\begin{equation*}
   g_n = \sum_{k=1}^{n} \binom{-3n}{k-1} \frac{(k-1)!}{n!} B_{n,k}(1!f_1, 2!f_2, \dots),
\end{equation*}
where $B_{n,k}$ is the $(n,k)$-th partial Bell polynomial. 

\smallskip
It is known that $f_n = \dfrac{3(2n-1)!2^n}{(n-1)!(n+2)!}$ for $n\ge 1$, giving the sequence (see \cite[A000257]{oeis})
\[ 1, 3, 12, 56, 288, 1584, 9152, 54912, 339456, 2149888,\dots, \]
and by means of the above formula for $g_n$, we get the sequence (see \cite[A298358]{oeis})
\[ 1, 0, 0, 1, 0, 3, 7, 15, 63, 168, 561, 1881, 6110, 21087, 72174,\dots \]
The asymptotic behavior of the sequence $(g_n)_{n\in\mathbb{N}}$ has been recently determined by Noy, Requil\'e, and Ru\'e \cite{NRR24}, who show that the generating function $G(x)$ is algebraic and derive a precise asymptotic estimate for $g_n$.

\smallskip
The Bell transform connection between $f$ and $g$ implies (see Birmajer et al.~\cite{BGMW}) that 
\begin{quote}
There is a bijection between the set of rooted bicubic planar maps on $2n$ vertices and the set of Dyck paths of semilength $3n$ having ascents of length multiples of 3, such that each $3j$-ascent may be colored in $g_j$ different ways. 
\end{quote}

One of the goals of this paper is to provide a combinatorial proof of this statement and give a constructive way to assemble the rooted bicubic planar maps using the $3$-connected elements as building blocks (primitives). This is done in Section~\ref{sec:construction}, following a section on preliminary definitions and basic notation.

In Section~\ref{sec:primitives}, we discuss the primitive maps and prove that, for $n\ge 4$, {\em every} 3-connected bicubic planar map on $2n$ vertices can be constructed from a smaller one by means of two simple insertion operations where either 4 or 6 vertices are inserted into a primitive element to obtain a larger one. This result has been previously discussed, for example in work by Horev et al. \cite{HKKN12}, but we give a straightforward proof here. We also discuss rootings and give lower bounds for the minimal number of possible rootings. In addition, for every $n\ge 11$, we use our insertion operations on prism graphs to give examples of 3-connected bicubic planar maps on $2n$ vertices having a maximal number of $6n$ distinct rootings. There is only one such map on $18$ vertices ($n=9$), see Figure~\ref{fig:primitive18}, and no such map for $n=10$ or $n<9$.

Finally, in the appendix, we illustrate our construction for all rooted bicubic planar maps on 6 vertices and provide a larger example. 

\section{Preliminaries}
\label{sec:prelim}

In this section we review some definitions and establish some conventions for dealing with the main objects of interest for this article. 

\subsection{Rooted cubic planar maps}
A {\em rooted planar map} is a closed 2-cell embedding of a connected graph in a sphere with a specified directed edge, called the {\em root edge}, assigned with an orientation. The face that lies to the right of the root edge while following its orientation is the {\em root face}, and the vertex at the tail-end of the root is the {\em root vertex}. A planar map is {\em bicubic} if every vertex has degree 3 and if all of its vertices can be bipartition into dark or white colors in such a way that adjacent vertices are assigned different colors. We adopt the convention to let the root vertex be dark.

\begin{figure}[ht]
\begin{tikzpicture}
\begin{scope}
\node[regular polygon, regular polygon sides=4, minimum size=60] at (0,0) (A) {};
\draw[poly] (A.corner 2) .. controls +(45:12pt) and +(135:12pt) .. (A.corner 1);
\draw[poly] (A.corner 2) .. controls +(-45:12pt) and +(-135:12pt) .. (A.corner 1);
\draw[poly] (A.corner 3) .. controls +(45:12pt) and +(135:12pt) .. (A.corner 4);
\draw[poly] (A.corner 3) .. controls +(-45:12pt) and +(-135:12pt) .. (A.corner 4);
\draw[poly] (A.corner 1) -- (A.corner 4);
\draw[root] (A.corner 2) -- (A.corner 3);
\foreach \i in {1,3} \node[v0] at (A.corner \i) {};
\foreach \i in {2,4} \node[v1] at (A.corner \i) {};
\end{scope}
\begin{scope}[xshift=90]
\node[regular polygon, regular polygon sides=4, minimum size=60] at (0,0) (A) {};
\draw[poly] (A.corner 2) .. controls +(45:12pt) and +(135:12pt) .. (A.corner 1);
\draw[poly] (A.corner 2) .. controls +(-45:12pt) and +(-135:12pt) .. (A.corner 1);
\draw[poly] (A.corner 3) .. controls +(45:12pt) and +(135:12pt) .. (A.corner 4);
\draw[poly] (A.corner 3) .. controls +(-45:12pt) and +(-135:12pt) .. (A.corner 4);
\draw[poly] (A.corner 1) -- (A.corner 4);
\draw[poly] (A.corner 2) -- (A.corner 3);
\draw[root] (A.corner 2) .. controls +(-45:12pt) and +(-135:12pt) .. (A.corner 1);
\foreach \i in {1,3} \node[v0] at (A.corner \i) {};
\foreach \i in {2,4} \node[v1] at (A.corner \i) {};
\end{scope}
\begin{scope}[xshift=180]
\node[regular polygon, regular polygon sides=4, minimum size=60] at (0,0) (A) {};
\draw[poly] (A.corner 2) .. controls +(45:12pt) and +(135:12pt) .. (A.corner 1);
\draw[poly] (A.corner 2) .. controls +(-45:12pt) and +(-135:12pt) .. (A.corner 1);
\draw[poly] (A.corner 3) .. controls +(45:12pt) and +(135:12pt) .. (A.corner 4);
\draw[poly] (A.corner 3) .. controls +(-45:12pt) and +(-135:12pt) .. (A.corner 4);
\draw[poly] (A.corner 1) -- (A.corner 4);
\draw[poly] (A.corner 2) -- (A.corner 3);
\draw[root] (A.corner 2) .. controls +(45:12pt) and +(135:12pt) .. (A.corner 1);
\foreach \i in {1,3} \node[v0] at (A.corner \i) {};
\foreach \i in {2,4} \node[v1] at (A.corner \i) {};
\end{scope}
\end{tikzpicture}
\caption{Rooted bicubic planar maps on 4 vertices.}
\label{fig:4vertices}
\end{figure}
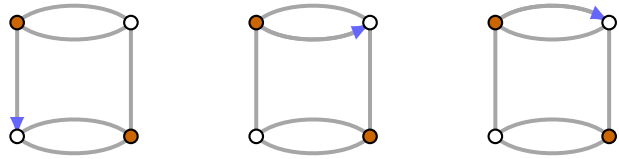

A bicubic planar map on $2n$ vertices has $3n$ edges and $n+2$ faces. We call a bicubic planar map {\em primitive} if it is 3-connected (short for 3-vertex-connected), that is, if the removal of any two vertices and their incident edges does not disconnect the map. It is known that a primitive bicubic planar map remains connected whenever any two edges are removed. There are no primitive bicubic planar maps on 4, 6, or 10 vertices. All primitive elements on 2, 8, and 12 vertices are shown in Figure~\ref{fig:basicprimitives}.

\subsection{Labeling algorithm}
For later purposes, we need an algorithm to label the edges of our primitive elements in a consistent manner. This will be done by means of a counter-clockwise walk (visiting all edges of the map, twice) similar to the one used to bijectively connect rooted planar maps and certain non-crossing Dyck word shuffles (see \cite[Sec.~1]{BGGP16}). First, we construct a {\em road} (wide path) that contains all vertices (and some edges) of the map according to the following algorithm:
\begin{itemize}
\item Start constructing the road at the root vertex and continue along the root edge to the adjacent vertex.
\item At each subsequent vertex, continue along the right-most edge whose endpoint has not been visited yet. In the absence of such an edge, turn around and travel back along the road (pausing the construction) until you reach the first vertex with an unvisited neighbor. Continue the construction along that edge; see Figure~\ref{fig:14primitive}.
\item Stop when all vertices are part of the road.
\end{itemize}

\begin{figure}[ht]
\begin{tikzpicture}
\node[shape=rounded rectangle, inner xsep=74, inner ysep=40, poly] at (3,0) (H) {};
\node[regular polygon, regular polygon sides=4, minimum size=40, rotate=45, poly] at (1.8,0) (A) {};
\node[regular polygon, regular polygon sides=4, minimum size=40, rotate=45, poly] at (4.2,0) (B) {};
\draw[poly] (A.corner 1) -- (H.north west);
\draw[poly] (A.corner 2) -- (H.west);
\draw[poly] (A.corner 3) -- (H.south west);
\draw[poly] (A.corner 4) -- (B.corner 2);
\draw[poly] (B.corner 1) -- (H.north east);
\draw[poly] (B.corner 4) -- (H.east);
\draw[root] (H.south east) -- (B.corner 3);
\foreach \i in {1,3} \node[v1] at (A.corner \i) {};
\foreach \i in {2,4} \node[v0] at (A.corner \i) {};
\foreach \i in {1,3} \node[v0] at (B.corner \i) {};
\foreach \i in {2,4} \node[v1] at (B.corner \i) {};
\foreach \vertex in {east,north west,south west} \node[v0] at (H.\vertex) {};
\foreach \vertex in {north east,west,south east} \node[v1] at (H.\vertex) {};
\draw[road] 
	(H.south east) -- (B.corner 3) -- (B.corner 4) -- 
	(H.east) .. controls +(95:23pt) and +(-5:23pt) .. (H.north east) -- 
	(H.north west) .. controls +(185:23pt) and +(85:23pt) .. (H.west) 
	.. controls +(275:23pt) and +(175:23pt) .. (H.south west) --
	(A.corner 3) -- (A.corner 4) -- (B.corner 2) -- (B.corner 1);
\draw[road, blue] (A.corner 4) -- (A.corner 1) -- (A.corner 2);
\end{tikzpicture}
\caption{Road constructed with back traveling.}
\label{fig:14primitive}
\end{figure}
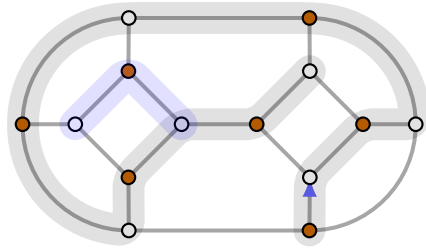

Once the road is finished, assign to the root edge the label `1' and walk counter-clockwise around the boundary of the created road, starting at the root face. Whenever an edge of the map is reached for the first time (traveling parallel to it or crossing it), increase the labeling by one until all edges have been labeled.

This algorithm gives a unique labeling for each primitive rooted bicubic planar map. More examples are shown in Figure~\ref{fig:basicprimitives}.

\begin{figure}[ht]
\begin{tikzpicture}
\draw[poly] (0,0) .. controls +(45:23pt) and +(135:23pt) .. (2,0);
\draw[root] (0,0) .. controls +(-45:23pt) and +(-135:23pt) .. (2,0);
\draw[poly] (0,0) -- (2,0);
\node[v1] at (0,0) {}; \node[v0] at (2,0) {};
\draw[road] (0,0) .. controls +(-45:23pt) and +(-135:23pt) .. (2,0);
\node[label,below=10pt] at (1,0) {\tiny $1$};
\node[label,above=10pt] at (1,0) {\tiny $2$};
\node[label,above=-2pt] at (1,0) {\tiny $3$};
\end{tikzpicture}

\vskip10pt

\begin{tikzpicture}
\node[regular polygon, regular polygon sides=4, minimum size=40, poly] at (0,0) (A) {};
\node[regular polygon, regular polygon sides=4, minimum size=90, poly] at (0,0) (B) {};
\foreach \i in {1,...,4} {\draw[poly] (A.corner \i) -- (B.corner \i);}
\draw[root] (B.corner 3) -- (B.corner 4);
\foreach \i in {1,3}{
	\node[v0] at (A.corner \i) {};
	\node[v1] at (B.corner \i) {};
}
\foreach \i in {2,4}{
	\node[v1] at (A.corner \i) {};
	\node[v0] at (B.corner \i) {};
}
\draw[road] 
 (B.corner 3) -- (B.corner 4) -- (B.corner 1) -- (B.corner 2) -- 
 (A.corner 2) -- (A.corner 3) -- (A.corner 4) -- (A.corner 1);
\node[label,below=30] at (0,0) {\tiny $1$};
\node[label,right=30] at (0,0) {\tiny $2$};
\node[label,above=30] at (0,0) {\tiny $3$};
\node[label,left=30] at (0,0) {\tiny $4$};
\node[label,above=18,left=18] at (0,0) {\tiny $5$};
\node[label,left=12] at (0,0) {\tiny $6$};
\node[label,below=18,left=17] at (0,0) {\tiny $7$};
\node[label,below=12] at (0,0) {\tiny $8$};
\node[label,below=18,right=18] at (0,0) {\tiny $9$};
\node[label,right=12] at (0,0) {\tiny $10$};
\node[label,above=18,right=17] at (0,0) {\tiny $11$};
\node[label,above=13] at (0,0) {\tiny $12$};
\end{tikzpicture}

\vskip14pt

\begin{tikzpicture}
\begin{scope}
\node[regular polygon, regular polygon sides=6, minimum size=40, poly] at (0,0) (A) {};
\node[regular polygon, regular polygon sides=6, minimum size=90, poly] at (0,0) (B) {};
\foreach \i in {1,...,6} {\draw[poly] (A.corner \i) -- (B.corner \i);}
\draw[root] (B.corner 4) -- (B.corner 5);
\foreach \i in {1,3,5}{
	\node[v1] at (A.corner \i) {};
	\node[v0] at (B.corner \i) {};
}
\foreach \i in {2,4,6}{
	\node[v0] at (A.corner \i) {};
	\node[v1] at (B.corner \i) {};
}
\draw[road] 
 (B.corner 4) -- (B.corner 5) -- (B.corner 6) -- (B.corner 1) -- (B.corner 2) -- (B.corner 3) --
 (A.corner 3) -- (A.corner 4) -- (A.corner 5) -- (A.corner 6) -- (A.corner 1) -- (A.corner 2);
\node[label,below=37] at (0,0) {\tiny $1$};
\node[label,below=20, right=32] at (0,0) {\tiny $2$};
\node[label,above=20, right=32] at (0,0) {\tiny $3$};
\node[label,above=37] at (0,0) {\tiny $4$};
\node[label,above=20, left=32] at (0,0) {\tiny $5$};
\node[label,below=20, left=32] at (0,0) {\tiny $6$};
\node[label,below=4, left=24] at (0,0) {\tiny $7$};
\node[label,below=10, left=13] at (0,0) {\tiny $8$};
\node[label,below=25, left=13] at (0,0) {\tiny $9$};
\node[label,below=15] at (0,0) {\tiny $10$};
\node[label,below=30, right=4] at (0,0) {\tiny $11$};
\node[label,below=10, right=12] at (0,0) {\tiny $12$};
\node[label,below=4, right=24] at (0,0) {\tiny $13$};
\node[label,above=10, right=12] at (0,0) {\tiny $14$};
\node[label,above=25, right=12] at (0,0) {\tiny $15$};
\node[label,above=15] at (0,0) {\tiny $16$};
\node[label,above=30, left=3] at (0,0) {\tiny $17$};
\node[label,above=10, left=13] at (0,0) {\tiny $18$};
\end{scope}
\begin{scope}[xshift=110]
\node[regular polygon, regular polygon sides=6, minimum size=40, poly] at (0,0) (A) {};
\node[regular polygon, regular polygon sides=6, minimum size=90, poly] at (0,0) (B) {};
\foreach \i in {1,...,6} {\draw[poly] (A.corner \i) -- (B.corner \i);}
\draw[root] (B.corner 4) -- (A.corner 4);
\foreach \i in {1,3,5}{
	\node[v1] at (A.corner \i) {};
	\node[v0] at (B.corner \i) {};
}
\foreach \i in {2,4,6}{
	\node[v0] at (A.corner \i) {};
	\node[v1] at (B.corner \i) {};
}
\draw[road] 
 (B.corner 4) -- (A.corner 4) -- (A.corner 5) --
 (B.corner 5) -- (B.corner 6) -- (B.corner 1) -- (B.corner 2) -- (B.corner 3) --
 (A.corner 3) --  (A.corner 2) -- (A.corner 1) -- (A.corner 6);
\node[label,below=29, left=6] at (0,0) {\tiny $1$};
\node[label,below=15] at (0,0) {\tiny $2$};
\node[label,below=29, right=6] at (0,0) {\tiny $3$};
\node[label,below=29] at (0,0) {\tiny $4$};
\node[label,below=20, right=32] at (0,0) {\tiny $5$};
\node[label,above=20, right=32] at (0,0) {\tiny $6$};
\node[label,above=37] at (0,0) {\tiny $7$};
\node[label,above=20, left=32] at (0,0) {\tiny $8$};
\node[label,below=20, left=32] at (0,0) {\tiny $9$};
\node[label,below=4, left=24] at (0,0) {\tiny $10$};
\node[label,below=10, left=13] at (0,0) {\tiny $11$};
\node[label,above=7, left=2] at (0,0) {\tiny $12$};
\node[label,above=7] at (0,0) {\tiny $13$};
\node[label,above=7, right=2] at (0,0) {\tiny $14$};
\node[label,below=7, right=2] at (0,0) {\tiny $15$};
\node[label,below=4, right=24] at (0,0) {\tiny $16$};
\node[label,above=25, right=12] at (0,0) {\tiny $17$};
\node[label,above=30, left=3] at (0,0) {\tiny $18$};
\end{scope}
\begin{scope}[xshift=220]
\node[regular polygon, regular polygon sides=6, minimum size=40, poly] at (0,0) (A) {};
\node[regular polygon, regular polygon sides=6, minimum size=90, poly] at (0,0) (B) {};
\foreach \i in {1,...,6} {\draw[poly] (A.corner \i) -- (B.corner \i);}
\draw[root] (B.corner 4) -- (B.corner 3);
\foreach \i in {1,3,5}{
	\node[v1] at (A.corner \i) {};
	\node[v0] at (B.corner \i) {};
}
\foreach \i in {2,4,6}{
	\node[v0] at (A.corner \i) {};
	\node[v1] at (B.corner \i) {};
}
\draw[road] 
 (B.corner 4) -- (B.corner 3) --
 (A.corner 3) --  (A.corner 4) -- (A.corner 5) --
 (B.corner 5) -- (B.corner 6) -- (B.corner 1) -- (B.corner 2) --
 (A.corner 2) -- (A.corner 1) -- (A.corner 6);
\node[label,below=16, left=25] at (0,0) {\tiny $1$};
\node[label,below=4, left=24] at (0,0) {\tiny $2$};
\node[label,below=10, left=13] at (0,0) {\tiny $3$};
\node[label,below=25, left=13] at (0,0) {\tiny $4$};
\node[label,below=15] at (0,0) {\tiny $5$};
\node[label,below=30, right=7] at (0,0) {\tiny $6$};
\node[label,below=29] at (0,0) {\tiny $7$};
\node[label,below=20, right=32] at (0,0) {\tiny $8$};
\node[label,above=20, right=32] at (0,0) {\tiny $9$};
\node[label,above=37] at (0,0) {\tiny $10$};
\node[label,above=20, left=32] at (0,0) {\tiny $11$};
\node[label,above=26, left=13] at (0,0) {\tiny $12$};
\node[label,above=10, left=13] at (0,0) {\tiny $13$};
\node[label,above=7] at (0,0) {\tiny $14$};
\node[label,above=7, right=2] at (0,0) {\tiny $15$};
\node[label,below=7, right=2] at (0,0) {\tiny $16$};
\node[label,below=4, right=24] at (0,0) {\tiny $17$};
\node[label,above=25, right=12] at (0,0) {\tiny $18$};
\end{scope}
\end{tikzpicture}
\caption{Edge-labeled primitive rooted bicubic planar maps on 2, 8, and 12 vertices, together with their corresponding road used for labeling.}
\label{fig:basicprimitives}
\end{figure}
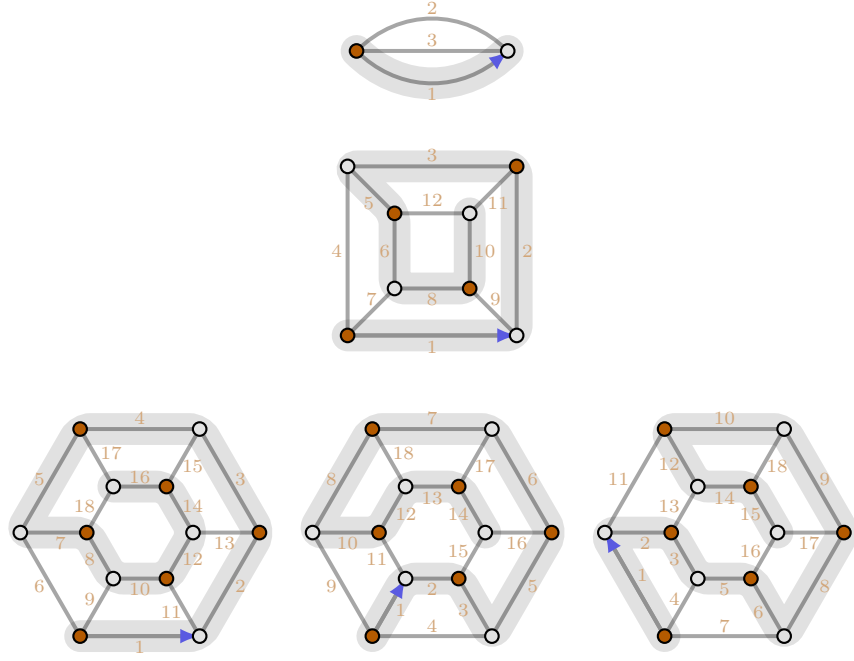

\subsection{Colored/decorated Dyck paths}
A {\em Dyck path} of semilength $n$ (or $n$-Dyck path) is a lattice path from $(0,0)$ to $(2n,0)$ consisting of steps $\U=(1,1)$ and $\D=(1,-1)$, never going below the $x$-axis. Clearly, every Dyck path can be encoded by a {\em Dyck word} over the alphabet $\{\U,\D\}$. We will freely pass from paths to words and back.  

Given an $n$-Dyck path, we will label its $\U$-steps as follows. If the first ascent has $j_1$ $\U$-steps, we label them with the numbers $1$ through $j_1$ in reversed order. Iteratively, if the $i$th ascent has $j_i$ $\U$-steps, we label them with the numbers $j_{i-1}+1$ through $j_{i-1}+j_i$ in reversed order. For example, the Dyck path $\U^3\D^2\U^3\D^4$ will be labeled as shown in Figure~\ref{fig:labeledDyckPath},

\begin{figure}[ht]
\begin{tikzpicture}[scale=0.5]
\dyckpath{0,0}{6}{1,1,1,0,0,1,1,1,0,0,0,0}
\begin{scope}
\foreach \l in {1,2,3} {
  \node[label,left=3pt] at (4-\l,3.95-\l) {\scriptsize \l}; 
}
\foreach \l in {4,5,6} {
  \node[label,left=3pt] at (12-\l,7.95-\l) {\scriptsize \l}; 
}
\end{scope}
\end{tikzpicture}
\caption{Labeled Dyck path.}
\label{fig:labeledDyckPath}
\end{figure}
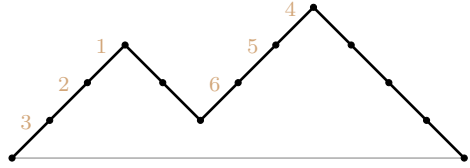

Let $g=(g_n)_{n\in\mathbb{N}}$ be the sequence $1, 0, 0, 1, 0, 3, 7, 15, 63, 168,\dots$, that enumerates the sets of 3-connected rooted bicubic planar maps on $2n$ vertices. Following the terminology from \cite{BGMW}, in this paper we are concerned with the set $\mathfrak{D}^{\mathbf{g}}_n(3,0)$ of Dyck words of semilength $3n$ created from strings of the form 
\[ \D, \U^3\D, \U^6\D, \U^9\D, \U^{12}\D,\dots \]
such that each building block $\U^{3j}\D$ may be colored in $g_j$ different ways. Equivalently, we could say that each $\U^{3j}\D$, $j\ge 1$, may be decorated with any of the $g_j$ 3-connected rooted bicubic planar maps on $2j$ vertices. Notice that, since $g_2=g_3=g_5=0$, we have that an element of $\mathfrak{D}^{\mathbf{g}}_n(3,0)$ will never have ascents of length 6, 9, or 15.

\section{Constructing rooted bicubic planar maps}
\label{sec:construction}

We start with a gluing operation suitable for our purposes. Let $M$ and $N$ be rooted bicubic planar maps with $m$ and $n$ edges, respectively. Suppose $\ell$ is an edge in $M$ and $\varrho$ is the root edge of $N$. We let $M \merge_{\ell}\, N$ denote the rooted bicubic planar map with $m+n$ edges obtained by gluing $N$ to the edge labeled $\ell$ on $M$ as follows:
\begin{itemize}[leftmargin=20pt]
\item Choose the representaion of $N$ such that its root face is the infinite face. 
\item In $M$, mark the face to the right of $\ell$ when going from the dark to the white vertex, and place $N$ on that face such that edges $\ell$ and $\varrho$ are ``parallel'' and their corresponding vertices have opposite polarity. 
\item Cut edges $\ell$ and $\varrho$, keeping the labels on the half edges connected to the dark vertices, and connect the two maps by gluing the half edges (dark vertex of $M$ to white vertex of $N$ for the one edge, and accordingly for the other edge) as illustrated in Figure~\ref{fig:gluingOp}.
\end{itemize}

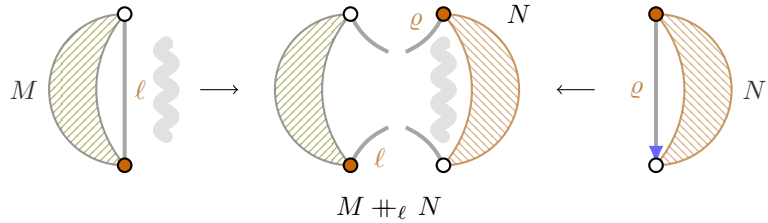
\begin{figure}[h]
\begin{tikzpicture}
\begin{scope}
\path[draw=gray,thick,opacity=0.8, pattern=north east lines, pattern color=olive] (0,0) to [bend left=40] (0,2) to [bend right=45] (-1,1) node[left=0] {\small $M$} to [bend right=45] (0,0);
\draw[poly] (0,0) -- (0,2);
\node[v1] at (0,0) {}; \node[v0] at (0,2) {};
\node[brown,right=0pt] at (0,1) {\small $\ell$};
\draw[scribble] (0.55,0.38) -- (0.55,1.62);
\draw[->] (1,1) -- (1.5,1);
\end{scope}
\begin{scope}[xshift=85]
\path[draw=gray,thick,opacity=0.8, pattern=north east lines, pattern color=olive] (0,0) to [bend left=40] (0,2) to [bend right=45] (-1,1) to [bend right=45] (0,0);
\draw[poly] plot [smooth, tension=1] coordinates {(0,0) (0.2,0.3) (0.5,0.5)};
\draw[poly] plot [smooth, tension=1] coordinates {(0,2) (0.2,1.7) (0.5,1.5)};
\node[v1] at (0,0) {}; \node[v0] at (0,2) {};
\node[brown,right=5pt] at (0,0.1) {\small $\ell$};
\end{scope}
\begin{scope}[xshift=120]
\path[draw=brown, thick, opacity=0.8, pattern=north west lines, pattern color=brown] (0,0) to [bend right=40] (0,2) to [bend left=45] (1,1) to [bend left=45] (0,0);
\draw[poly] plot [smooth, tension=1] coordinates {(0,0) (-0.2,0.3) (-0.5,0.5)};
\draw[poly] plot [smooth, tension=1] coordinates {(0,2) (-0.2,1.7) (-0.5,1.5)};
\node[v0] at (0,0) {}; \node[v1] at (0,2) {};
\node at (1,2) {\small $N$};
\node[brown,left=4pt] at (0,1.9) {\small $\varrho$};
\draw[scribble] (0,0.38) -- (0,1.62);
\draw[<-] (1.5,1) -- (2,1);
\end{scope}
\begin{scope}[xshift=200]
\path[draw=brown, thick, opacity=0.8, pattern=north west lines, pattern color=brown] (0,0) to [bend right=40] (0,2) to [bend left=45] (1,1) node[right=0] {\small $N$} to [bend left=45] (0,0);
\draw[root] (0,2) -- (0,0);
\node[v0] at (0,0) {}; \node[v1] at (0,2) {};
\node[brown,left=1pt] at (0,1) {\small $\varrho$};
\end{scope}
\node[below=8pt] at (3.5,0) {\small $M \merge_{\ell}\, N$};
\end{tikzpicture}
\caption{Gluing operation.}
\label{fig:gluingOp}
\end{figure}

The resulting map, denoted by $M \merge_{\ell}\, N$, is a rooted bicubic planar map with $m+n$ edges whose root vertex is that of $M$.

\begin{lemma} \label{lem:decompLemma}
Let $M$ be a non-primitive rooted bicubic planar map and let $e$ be the first edge that is part of a 2-edge cut set following the road construction from the labeling algorithm. Then $M$ has a 2-edge cut set $\{e,e'\}$ that decomposes $M$ into two rooted bicubic planar maps $M_1$ and $M_2$, where $M_1$ contains the root vertex of $M$ and is minimal with this property.
\end{lemma}

\begin{proof}
Let $F$ and $F'$ be the two faces adjacent to $e$, and let $C$ be the set of edges shared by $F$ and $F'$. Note that any pair of edges in $C$ is non-adjacent and forms a 2-edge cut set for $M$. Let $M'$ be the connected component of $M-C$ that contains the root vertex, and let $M''$ be the union of the remaining components. By construction, edge $e$ connects $M'$ and $M''$. We let $e'$ be the other edge in $C$ that also connects $M'$ and $M''$.

Cut the edges $e$ and $e'$, and let $M_1$ be the rooted bicubic planar map obtained from $M'$ by gluing the half edges connected to it. Similarly, let $M_2$ be the rooted bicubic planar map obtained from the other connected component of $M-\{e,e'\}$ by gluing the remaining half edges and making this new edge the root of $M_2$.

This gives the desired decomposition of $M$. See Figure~\ref{fig:decomposition} for an illustration.
\end{proof}

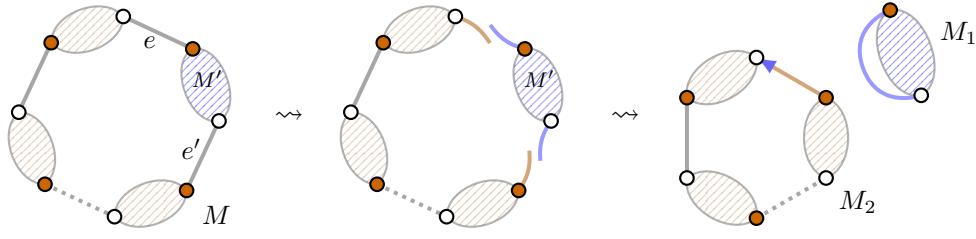
\begin{figure}[ht]
\begin{tikzpicture}
\tikzstyle{shaded}=[draw=gray,thick,opacity=0.6, pattern=north east lines]
\begin{scope}
\node[regular polygon, regular polygon sides=8, minimum size=75, rotate=20] at (0,0) (A) {};
\path[shaded, pattern color=blue] (A.corner 7) to [bend right=80] (A.corner 8)  to [bend right=80]  (A.corner 7);
\path[shaded, pattern color=brown] (A.corner 1) to [bend right=70] (A.corner 2)  to [bend right=70]  (A.corner 1);
\path[shaded, pattern color=brown] (A.corner 3) to [bend right=70] (A.corner 4)  to [bend right=70]  (A.corner 3);
\path[shaded, pattern color=brown] (A.corner 5) to [bend right=70] (A.corner 6)  to [bend right=70]  (A.corner 5);
\node[below right=2pt] at (A.corner 6) {\small $M$};
\node at (1.17,0.5) {\scriptsize $M'$};
\draw[poly] (A.corner 8) -- (A.corner 1) node[black, below right=4pt]{\small $e$};
\draw[poly] (A.corner 2) -- (A.corner 3);
\draw[poly] (A.corner 6)  -- (A.corner 7) node[black, below left=3pt]{\small $e'$};
\draw[poly,dotted] (A.corner 4) -- (A.corner 5);
\foreach \i in {1,3,5,7}{
	\node[v0] at (A.corner \i) {};
}
\foreach \i in {2,4,6,8}{
	\node[v1] at (A.corner \i) {};
}
\end{scope}
\node at (2.25,0) {$\leadsto$};
\node at (6.7,0) {$\leadsto$};

\begin{scope}[xshift=125]
\node[regular polygon, regular polygon sides=8, minimum size=75, rotate=20] at (0,0) (A) {};
\path[shaded, pattern color=blue] (A.corner 7) to [bend right=80] (A.corner 8)  to [bend right=80]  (A.corner 7);
\path[shaded, pattern color=brown] (A.corner 1) to [bend right=70] (A.corner 2)  to [bend right=70]  (A.corner 1);
\path[shaded, pattern color=brown] (A.corner 3) to [bend right=70] (A.corner 4)  to [bend right=70]  (A.corner 3);
\path[shaded, pattern color=brown] (A.corner 5) to [bend right=70] (A.corner 6)  to [bend right=70]  (A.corner 5);
\node at (1.17,0.5) {\scriptsize $M'$};
\draw[poly,blue!40] (A.corner 8)  to[bend left=20] ++(-13pt,9pt);
\draw[poly,brown!70] (A.corner 1)  to[bend left=20] ++(13pt,-10pt);
\draw[poly,brown!70] (A.corner 6)  to[bend right=20] ++(4pt,15pt);
\draw[poly,blue!40] (A.corner 7)  to[bend right=20] ++(-4pt,-15pt);
\draw[poly] (A.corner 2) -- (A.corner 3);
\draw[poly,dotted] (A.corner 4) -- (A.corner 5);
\foreach \i in {1,3,5,7}{
	\node[v0] at (A.corner \i) {};
}
\foreach \i in {2,4,6,8}{
	\node[v1] at (A.corner \i) {};
}
\end{scope}
\begin{scope}[xshift=240,yshift=-8]
\node[regular polygon, regular polygon sides=6, minimum size=60, rotate=30] at (0,0) (A) {};
\node[below right=1pt] at (A.corner 5) {\small $M_2$};
\path[shaded, pattern color=brown] (A.corner 1) to [bend right=70] (A.corner 2)  to [bend right=70]  (A.corner 1);
\path[shaded, pattern color=brown] (A.corner 3) to [bend right=70] (A.corner 4)  to [bend right=70]  (A.corner 3);
\path[shaded, pattern color=brown] (A.corner 5) to [bend right=70] (A.corner 6)  to [bend right=70]  (A.corner 5);
\draw[root,brown!70] (A.corner 6) -- (A.corner 1);
\draw[poly] (A.corner 2) -- (A.corner 3);
\draw[poly,dotted] (A.corner 4) -- (A.corner 5);
\foreach \i in {1,3,5}{
	\node[v0] at (A.corner \i) {};
}
\foreach \i in {2,4,6}{
	\node[v1] at (A.corner \i) {};
}
\end{scope}

\begin{scope}[xshift=302,yshift=8,rotate=20]
\path[shaded, pattern color=blue] (0,0) to [bend right=80] (0,1.2)  to [bend right=80]  (0,0);
\draw[poly,blue!40] plot [smooth, tension=1.8] coordinates {(0,0) (-0.6,0.6) (0,1.2)};
\node at (0.75,0.6) {\small $M_1$};
\node[v0] at (0,0) {};
\node[v1] at (0,1.2) {};
\end{scope}

\end{tikzpicture}
\caption{Sketch for the proof of Lemma~\ref{lem:decompLemma}.}
\label{fig:decomposition}
\end{figure}

\begin{theorem} \label{thm:mainBijection}
There is a bijection between the set of rooted bicubic maps on $2n$ vertices and the set of Dyck paths of semilength $3n$ having ascents of length multiple of 3, and such that each $3j$-ascent may be colored in $g_j$ different ways. 
\end{theorem}

\begin{proof}
The idea of the proof is that every rooted bicubic planar map on $2n$ vertices can be uniquely constructed from 3-connected rooted bicubic planar maps (primitives) that are ``stored'' on the ascents of a Dyck path of semilength $3n$.

If $P$ is a Dyck path of the form $\U^{3n}\D^{3n}$, decorated by the primitive map $B$, then we let $\phi(P)=B$.
Suppose now that $P$ is a Dyck path in $\mathfrak{D}^{\mathbf{g}}_n(3,0)$ of the form
\[ \U^{j_1}\D^{\ell_1}\cdots \U^{j_k}\D^{\ell_k} \]
with $j_i\ge 1$, $k>1$, and ascents decorated by primitive maps $B_1,\dots,B_k$ with $j_1,\dots,j_k$ edges, respectively. We then construct a planar map $M_P=\phi(P)$ as follows:

\begin{itemize}[leftmargin=20pt]
\item Use the algorithms from Section~\ref{sec:prelim} to label the primitive maps and the ascents of $P$.
\item For the maximal descent following the $i$th peak of $P$, let $e_i$ be the label of the $\U$-step matching the final $\D$-step of that descent.
\item Merge the primitive maps $B_1,\dots,B_k$ by letting $M_1=B_1$, and iteratively,
\begin{equation*}
M_{j+1} = M_{j} \merge_{e_j}\, B_{j+1} \;\text{ for } j\in\{1,\dots,k-1\}.
\end{equation*}
Before each merge, increment the edge labels of $B_{j+1}$ by the number of edges in $M_j$.
\end{itemize}
We let $M_P=M_{k}$. It is a rooted bicubic planar map with $j_1+\dots+j_k$ edges.

For example, if $P$ is the decorated Dyck path in Figure~\ref{fig:3ThetaDyckPath}, where $\Theta=$~\tikz[scale=0.55,baseline=-3pt]{
\draw[poly] (0,0) .. controls +(45:23pt) and +(135:23pt) .. (2,0) -- (0,0);
\draw[root] (0,0) .. controls +(-45:23pt) and +(-135:23pt) .. (2,0);
\node[v1] at (0,0) {}; \node[v0] at (2,0) {};}\,, then we get $M_1 = \Theta\merge_{\bf 2}\,\Theta$ and $M_P = M_2 = M_1\merge_{\bf 4}\,\Theta$. See Figures \ref{fig:Theta*Theta} and \ref{fig:M1*Theta}. A much larger example with ascents longer than 3 is provided in the appendix.

\begin{figure}[h]
\begin{tikzpicture}[scale=0.5]
\dyckpath{0,0}{9}{1,1,1,0,0,1,1,1,0,1,1,1,0,0,0,0,0,0}
\begin{scope}
\draw[thick,dashed,cyan] (1.6,1.4) -- (4.5,1.4);
\draw[thick,dashed,cyan] (7.5,3.4) -- (8.5,3.4);
\foreach \l in {1,2,3} { \node[label,left=3pt] at (4-\l,3.95-\l) {\scriptsize \l}; }
\foreach \l in {4,5,6} { \node[label,left=3pt] at (12-\l,7.95-\l) {\scriptsize \l}; }
\foreach \l in {7,8,9} { \node[label,left=3pt] at (19-\l,12.95-\l) {\scriptsize \l}; }
\foreach \x/\y in {2/2,7/3,11/5} {\node[left=12pt] at (\x,\y) {$\Theta$};}
\end{scope}
\end{tikzpicture}
\caption{Decorated Dyck path $P$.}
\label{fig:3ThetaDyckPath}
\end{figure}
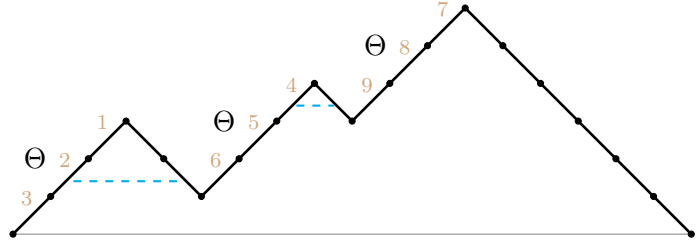

\begin{figure}[ht]
\begin{tikzpicture}[scale=0.75]
\begin{scope}
\draw[poly] (0,0) .. controls +(45:23pt) and +(135:23pt) .. (2,0);
\draw[poly] (0,0) .. controls +(-45:23pt) and +(-135:23pt) .. (2,0);
\draw[root] (0,0) -- (2,0);
\node[v1] at (0,0) {}; \node[v0] at (2,0) {};
\node[label,below=0pt] at (1,0) {\tiny $2$};
\node[label,above=8pt] at (1,0) {\tiny $3$};
\node[label,above=-2pt] at (1,0) {\tiny $1$};
\draw[scribble] (0.4,-0.75) -- (1.6,-0.75);
\draw[->] (1,-1) -- (1.5,-1.5);
\node[above=5pt] at (2,0) {\small $\Theta$};
\end{scope}
\begin{scope}[xshift=120]
\draw[poly] (2,0) .. controls +(135:23pt) and +(45:23pt) .. (0,0);
\draw[root] (0,0) .. controls +(-45:23pt) and +(-135:23pt) .. (2,0);
\draw[poly] (0,0) -- (2,0);
\node[v1] at (0,0) {}; \node[v0] at (2,0) {};
\node[label,below=0pt] at (1,0) {\tiny $4$};
\node[label,above=8pt] at (1,0) {\tiny $5$};
\node[label,above=-2pt] at (1,0) {\tiny $6$};
\draw[->] (1,-1) -- (0.5,-1.5);
\node[above=5pt] at (2,0) {\small $\Theta$};
\end{scope}
\begin{scope}[xshift=215]
\draw[poly] (0,0) .. controls +(45:23pt) and +(135:23pt) .. (2,0);
\draw[root] (0,0) .. controls +(-45:23pt) and +(-135:23pt) .. (2,0);
\draw[poly] (0,0) -- (2,0);
\node[v1] at (0,0) {}; \node[v0] at (2,0) {};
\node[label,below=0pt] at (1,0) {\tiny $1$};
\node[label,above=8pt] at (1,0) {\tiny $2$};
\node[label,above=-2pt] at (1,0) {\tiny $3$};
\draw[->] (-0.4,0) -- (-0.8,0);
\node[above=5pt] at (2,0) {\small $\Theta$};
\end{scope}
\begin{scope}[xshift=60,yshift=-100]
\draw[poly] (0,0) .. controls +(135:23pt) and +(-135:23pt) .. (0,2);
\draw[poly] (0.6,0) -- (0,0) -- (0,2) -- (0.6,2);
\draw[root] (0,0) -- (0,2);
\draw[scribble] (0.7,0.4) -- (0.7,1.6);
\node[v1] at (0,0) {}; \node[v0] at (0,2) {};
\node[label,below=-2pt] at (0.3,0) {\tiny $2$};
\node[label,left=8pt] at (0,1) {\tiny $3$};
\node[label,left=-2pt] at (0,1) {\tiny $1$};
\end{scope}
\begin{scope}[xshift=115,yshift=-100]
\draw[poly] (-0.6,0) -- (0,0) -- (0,2) -- (-0.6,2);
\draw[poly] (0,0) .. controls +(45:23pt) and +(-45:23pt) .. (0,2);
\draw[poly] (0,0) -- (0,2);
\node[v0] at (0,0) {}; \node[v1] at (0,2) {};
\node[label,right=8pt] at (0,1) {\tiny $5$};
\node[label,right=-2pt] at (0,1) {\tiny $6$};
\node[label,above=-2pt] at (-0.3,2) {\tiny $4$};
\end{scope}
\begin{scope}[xshift=200,yshift=-100]
\draw[->] (-1.7,1) -- (-1,1);
\draw[poly] (2,0) -- (0,0) .. controls +(135:15pt) and +(-135:15pt) .. (0,2) -- (2,2);
\draw[root] (0,0) .. controls +(45:15pt) and +(-45:15pt) .. (0,2);
\draw[poly] (2,0) .. controls +(45:15pt) and +(-45:15pt) .. (2,2);
\draw[poly] (2,0) .. controls +(135:15pt) and +(-135:15pt) .. (2,2);
\node[v1] at (0,0) {}; \node[v0] at (0,2) {};
\node[v0] at (2,0) {}; \node[v1] at (2,2) {};
\node[label,below=-2pt] at (1,0) {\tiny $2$};
\node[label,left=5pt] at (0,1) {\tiny $3$};
\node[label,right=-3pt] at (0,1) {\tiny $1$};
\node[label,right=5pt] at (2,1) {\tiny $5$};
\node[label,left=-3pt] at (2,1) {\tiny $6$};
\node[label,above=-2pt] at (1,2) {\tiny $4$};
\node[left=12pt] at (4.2,1) {\small $M_1$};
\end{scope}
\end{tikzpicture}
\caption{Construction of $M_1 = \Theta\merge_{\bf 2}\,\Theta$.}
\label{fig:Theta*Theta}
\end{figure}
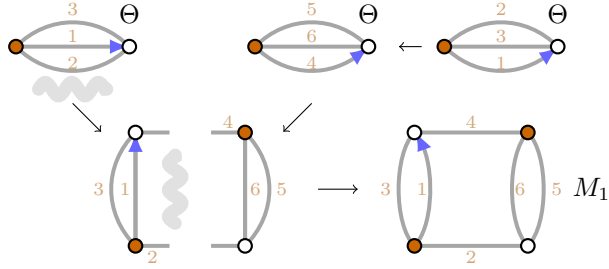

\begin{figure}[ht]
\begin{tikzpicture}[scale=0.7]
\begin{scope}
\draw[poly] (0,2) -- (0,0) .. controls +(45:15pt) and +(135:15pt) .. (2,0) -- (2,2);
\draw[poly] (0,0) .. controls +(-45:15pt) and +(-135:15pt) .. (2,0);
\draw[poly] (0,2) .. controls +(45:15pt) and +(135:15pt) .. (2,2);
\draw[root] (0,2) .. controls +(-45:15pt) and +(-135:15pt) .. (2,2);
\node[v0] at (0,0) {}; \node[v1] at (2,0) {};
\node[v1] at (0,2) {}; \node[v0] at (2,2) {};
\node[label,left=-2pt] at (0,1) {\tiny $2$};
\node[label,above=5pt] at (1,2) {\tiny $3$};
\node[label,above=-8pt] at (1,2) {\tiny $1$};
\node[label,right=-2pt] at (2,1) {\tiny $4$};
\node[label,above=5pt] at (1,0) {\tiny $6$};
\node[label,above=3pt] at (1,-0.5) {\tiny $5$};
\draw[scribble] (2.5,0.4) -- (2.5,1.7);
\draw[->] (1,-0.5) -- (1.5,-1);
\node[left=10pt] at (0,1) {\small $M_1$};
\end{scope}
\begin{scope}[xshift=150, yshift=20]
\draw[root] (2,0) .. controls +(135:23pt) and +(45:23pt) .. (0,0);
\draw[poly] (0,0) .. controls +(-45:23pt) and +(-135:23pt) .. (2,0);
\draw[poly] (0,0) -- (2,0);
\node[v0] at (0,0) {}; \node[v1] at (2,0) {};
\node[label,below=0pt] at (1,0) {\tiny $8$};
\node[label,above=8pt] at (1,0) {\tiny $7$};
\node[label,above=-2pt] at (1,0) {\tiny $9$};
\draw[->] (1,-1) -- (0.5,-1.5);
\node[above=5pt] at (2,0) {\small $\Theta$};
\end{scope}
\begin{scope}[xshift=40,yshift=-110]
\draw[poly] (0,2) -- (0,0) .. controls +(45:15pt) and +(135:15pt) .. (2,0) -- (2.6,0);
\draw[poly] (0,0) .. controls +(-45:15pt) and +(-135:15pt) .. (2,0);
\draw[poly] (0,2) .. controls +(45:15pt) and +(135:15pt) .. (2,2) -- (2.6,2);
\draw[root] (0,2) .. controls +(-45:15pt) and +(-135:15pt) .. (2,2);
\node[v0] at (0,0) {}; \node[v1] at (2,0) {};
\node[v1] at (0,2) {}; \node[v0] at (2,2) {};
\node[label,left=-2pt] at (0,1) {\tiny $2$};
\node[label,above=5pt] at (1,2) {\tiny $3$};
\node[label,above=-8pt] at (1,2) {\tiny $1$};
\node[label,right=-2pt] at (2,1) {\tiny $4$};
\node[label,above=5pt] at (1,0) {\tiny $6$};
\node[label,above=3pt] at (1,-0.5) {\tiny $5$};
\end{scope}
\begin{scope}[xshift=155,yshift=-110]
\draw[poly] (-0.6,0) -- (0,0) -- (0,2) -- (-0.6,2);
\draw[poly] (0,0) .. controls +(45:23pt) and +(-45:23pt) .. (0,2);
\draw[poly] (0,0) -- (0,2);
\node[v0] at (0,0) {}; \node[v1] at (0,2) {};
\node[label,right=8pt] at (0,1) {\tiny $8$};
\node[label,right=-2pt] at (0,1) {\tiny $9$};
\node[label,above=-2pt] at (-0.3,2) {\tiny $7$};
\end{scope}
\begin{scope}[xshift=315,yshift=-80]
\draw[->] (-4,0) -- (-3,0);
\node[regular polygon, regular polygon sides=6, minimum size=80] at (0,0) (A) {};
\draw[poly] (A.corner 3) .. controls +(105:15pt) and +(195:15pt) .. (A.corner 2);
\draw[root] (A.corner 3) .. controls +(15:15pt) and +(-75:15pt) .. (A.corner 2);
\draw[poly] (A.corner 4) .. controls +(45:15pt) and +(135:15pt) .. (A.corner 5);
\draw[poly] (A.corner 4) .. controls +(-45:15pt) and +(-135:15pt) .. (A.corner 5);
\draw[poly] (A.corner 6) .. controls +(75:15pt) and +(-15:15pt) .. (A.corner 1);
\draw[poly] (A.corner 6) .. controls +(165:15pt) and +(-105:15pt) .. (A.corner 1);
\foreach \i/\j in {1/2,3/4,5/6}{
	\draw[poly] (A.corner \i) -- (A.corner \j);
}
\foreach \i in {1,3,5}{
	\node[v1] at (A.corner \i) {};
}
\foreach \i in {2,4,6}{
	\node[v0] at (A.corner \i) {};
}
\node[left=20pt] at (4.3,1) {\small $M_P$};
\end{scope}
\end{tikzpicture}
\caption{Construction of $M_P = M_1\merge_{\bf 4}\,\Theta$.}
\label{fig:M1*Theta}
\end{figure}
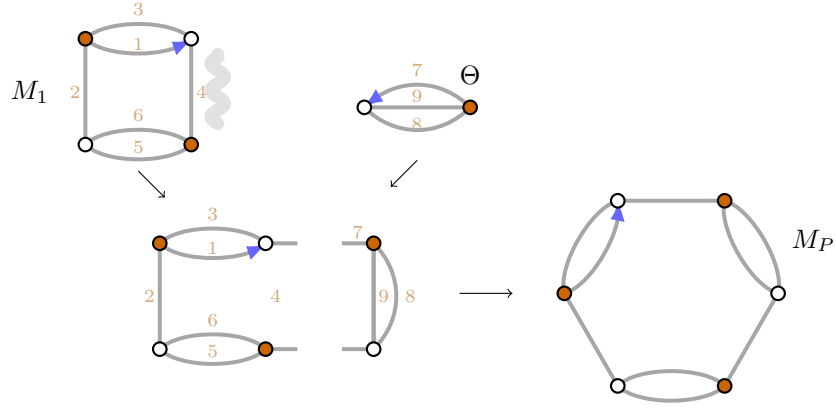

As mentioned in the introduction, we already know that the two sets in question are equinumerous. Thus, it is enough to prove that our map $\phi:P\mapsto M_P$ is surjective. We will do that by strong induction in $n$.

Note that the map $\Theta$ corresponds to the path $\U^3\D^3$, with ascent decorated by $\Theta$, and the three rooted bicubic planar maps on 4 vertices (shown in Figure~\ref{fig:4vertices}) correspond (from left to right) to the Dyck paths $\U^3\D\U^3\D^5$, $\U^3\D^2\U^3\D^4$, and $\U^3\D^3\U^3\D^3$, with $\Theta$-decorated ascents. In other words, these maps can be written as $\Theta\merge_{\bf 1}\,\Theta$, $\Theta\merge_{\bf 2}\,\Theta$, and $\Theta\merge_{\bf 3}\,\Theta$, respectively.

This takes care of the cases $n=1$ and $n=2$. Suppose now $M$ is a rooted bicubic planar map on $2n$ vertices with $n>2$. If $M$ is primitive, we let $P_M$ be the Dyck path $\U^{3n}\D^{3n}$ with $M$-decorated ascent. By definition, $\phi(P_M)=M$.

If $M$ is not primitive, we use Lemma~\ref{lem:decompLemma} to decompose $M$ into two smaller rooted bicubic planar maps $M_1$ and $M_2$, where $M_1$ has the root of $M$ and a distinguished edge created by the gluing of the two half edges from the cut set $\{e,e'\}$. By induction, there are Dyck paths $P_1$ and $P_2$ such that $\phi(P_i)=M_i$, $i=1,2$. Using the labeling induced by the map $\phi$, we let $a$ be the label of the distinguished edge in $M_1$. By the construction given in Lemma~\ref{lem:decompLemma}, label $a$ must be on the first ascent of $P_1$. 
Let $P$ be the Dyck path $P_1 \circ_a P_2$, obtained by inserting $P_2$ right after the down-step corresponding to the up-step labeled $a$. 

Let $3m$ be the length of the first ascent of $P_1$. If $a=3m$, then $P_1$ must be of the form $\U^{3m}\D^{3m}$, so we have $\phi(P) =  \phi(P_1) \merge_{a}\, \phi(P_2)$.

If $a<3m$, then from the $a$-th down-step in $P_1$, we go down until the last down-step of that descent is reached. Let $b$ be the label of the corresponding up-step. By construction, we have $a<b\le 3m$. Thus there are Dyck paths $P_1'$ and $P_1''$ (possibly the empty path) such that $P_1 = P_1'\circ_b P_1''$. 
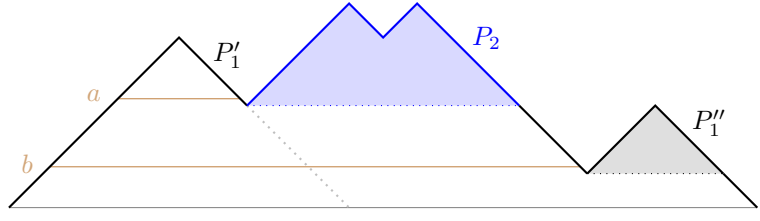
\begin{figure}[t]
\begin{tikzpicture}[scale=0.45]
\draw[help lines] (0,0) -- (22,0);
\node[label,left=0pt] at (3,3.3) {\small $a$};
\node[label,left=0pt] at (1,1.3) {\small $b$};
\draw[help lines,brown] (3.2,3.2) -- (6.8,3.2);
\draw[help lines,brown] (1.2,1.2) -- (16.8,1.2);
\draw[thick] (0,0) -- (5,5) -- (7,3);
\node[right=9pt] at (5,4.5) {\small $P_1'$};
\draw[thick,dotted,gray!50] (7,3) -- (10,0);
\draw[thick,blue,fill=blue!15] (7,3) -- (10,6) -- (11,5) -- (12,6) -- (15,3);
\draw[dotted,blue] (7,3) -- (15,3);
\node[blue,right=4pt] at (13,5) {\small $P_2$};
\draw[thick] (15,3) -- (17,1);
\draw[thick,fill=gray!25] (17,1) -- (19,3) -- (21,1);
\draw[dotted] (17,1) -- (21,1);
\node[right=10pt] at (19,2.5) {\small $P_1''$};
\draw[thick] (21,1) -- (22,0);
\end{tikzpicture}
\caption{Example of $(P_1'\circ_b P_1'') \circ_a P_2 = (P_1'\circ_a P_2) \circ_b P_1''$.}
\label{fig:DyckPathPartition}
\end{figure}
In this case, we have
\[ P = P_1 \circ_a P_2 = (P_1'\circ_b P_1'') \circ_a P_2 = (P_1'\circ_a P_2) \circ_b P_1'' 
  \quad\text{(see Figure~\ref{fig:DyckPathPartition}),}\]
and threfore, $\phi(P) = \phi((P_1'\circ_a P_2) \circ_b P_1'') = \phi((P_1'\circ_b P_1'') \circ_a P_2) = \phi(P_1)\merge_{a}\, \phi(P_2)$. 

In all cases, $\phi(P) = M_1\merge_{a}\, M_2 = M$, proving the surjectivity of $\phi$.
\end{proof}

Using the above construction, all 12 rooted bicubic planar maps on $6$ vertices can be generated from Dyck paths of the form $\U^3\D^i\U^3\D^j\U^3\D^k$ with $i,j,k\ge 1$ and $\Theta$-decorated ascents. They are listed in Table~\ref{tab:library9edges} with corresponding maps shown in Figure~\ref{fig:library9edges}, both in the appendix.

\section{Primitive elements}
\label{sec:primitives}

In this section, we will refer to 3-connected bicubic planar maps as {\em primitives}.

Let us recall that $\Theta=$~\tikz[scale=0.55,baseline=-3pt]{
\draw[poly] (0,0) .. controls +(45:23pt) and +(135:23pt) .. (2,0) -- (0,0);
\draw[poly] (0,0) .. controls +(-45:23pt) and +(-135:23pt) .. (2,0);
\node[v1] at (0,0) {}; \node[v0] at (2,0) {};}\, is the smallest primitive. Any primitive consisting of two $n$-faces and $n$ edges between these two faces will be referred to as an {\em $n$-prism}. A 4-prism is often called a {\em cube}. The cube can be obtained from $\Theta$ by a suitable insertion of 6 vertices and, in turn, the cube gives rise to the 6-prism by a 4-vertex insertion as shown in Figure~\ref{fig:8-12-vertices}.

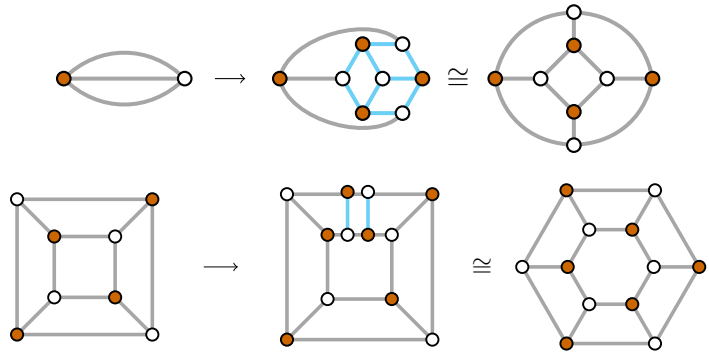
\begin{figure}[ht]
\begin{tikzpicture}[scale=0.8]
\begin{scope}
\draw[poly] (0,0) .. controls +(45:23pt) and +(135:23pt) .. (2,0);
\draw[poly] (0,0) .. controls +(-45:23pt) and +(-135:23pt) .. (2,0);
\draw[poly] (0,0) -- (2,0);
\node[v1] at (0,0) {}; \node[v0] at (2,0) {};
\draw[->] (2.5,0) -- (3,0);
\end{scope}
\begin{scope}[xshift=150]
\node[regular polygon, regular polygon sides=6, minimum size=30, poly,insertcol] at (0,0) (A) {};
\draw[poly] (-1.7,0) .. controls +(80:20pt) and +(130:20pt) .. (A.corner 1);
\draw[poly] (-1.7,0) .. controls +(-80:20pt) and +(-130:20pt) .. (A.corner 5);
\draw[poly] (-1.7,0) node[v1]{} -- (A.corner 3);
\foreach \i in {2,4,6}{
	\draw[poly,insertcol] (0,0) -- (A.corner \i);
}
\node[v0] at (0,0){};
\foreach \i in {1,3,5}{
	\node[v0] at (A.corner \i) {};
}
\foreach \i in {2,4,6}{
	\node[v1] at (A.corner \i) {};
}
\node at (1.25,0) {$\cong$};
\end{scope}
\begin{scope}[xshift=240]
\node[regular polygon, regular polygon sides=4, minimum size=25, poly, rotate=45] at (0,0) (A) {};
\draw[poly] (-1.3,0) .. controls +(80:42pt) and +(100:42pt) .. (1.3,0);
\draw[poly] (-1.3,0) .. controls +(-80:42pt) and +(-100:42pt) .. (1.3,0);
\draw[poly] (A.corner 1) -- (0,1.1) node[v0]{};
\draw[poly] (A.corner 2) -- (-1.3,0) node[v1]{};
\draw[poly] (A.corner 3) -- (0,-1.1) node[v0]{};
\draw[poly] (A.corner 4) -- (1.3,0) node[v1]{};
\foreach \i in {1,3}{
	\node[v1] at (A.corner \i) {};
}
\foreach \i in {2,4}{
	\node[v0] at (A.corner \i) {};
}
\end{scope}
\end{tikzpicture}

\medskip
\scalebox{0.9}{
\begin{tikzpicture}
\begin{scope}
\node[regular polygon, regular polygon sides=4, minimum size=36, poly] at (0,0) (A) {};
\node[regular polygon, regular polygon sides=4, minimum size=80, poly] at (0,0) (B) {};
\foreach \i in {1,...,4} {\draw[poly] (A.corner \i) -- (B.corner \i);}
\foreach \i in {1,3}{
	\node[v0] at (A.corner \i) {};
	\node[v1] at (B.corner \i) {};
}
\foreach \i in {2,4}{
	\node[v1] at (A.corner \i) {};
	\node[v0] at (B.corner \i) {};
}
\draw[->] (1.75,0) -- (2.25,0);
\end{scope}
\begin{scope}[xshift=115]
\node[regular polygon, regular polygon sides=4, minimum size=38, poly] at (0,0) (A) {};
\node[regular polygon, regular polygon sides=4, minimum size=86, poly] at (0,0) (B) {};
\foreach \i in {1,...,4} {\draw[poly] (A.corner \i) -- (B.corner \i);}
\foreach \i in {1,3}{
	\node[v0] at (A.corner \i) {};
	\node[v1] at (B.corner \i) {};
}
\foreach \i in {2,4}{
	\node[v1] at (A.corner \i) {};
	\node[v0] at (B.corner \i) {};
}
\draw[poly,insertcol] (A.corner 1)+(-0.35,0) node[v1]{} -- +(-0.35,0.63) node[v0]{};
\draw[poly,insertcol] (A.corner 1)+(-0.65,0) node[v0]{} -- +(-0.65,0.63) node[v1]{};
\node at (1.8,0) {$\cong$};
\end{scope}
\begin{scope}[xshift=220]
\node[regular polygon, regular polygon sides=6, minimum size=36, poly] at (0,0) (A) {};
\node[regular polygon, regular polygon sides=6, minimum size=74, poly] at (0,0) (B) {};
\foreach \i in {1,...,6} {\draw[poly] (A.corner \i) -- (B.corner \i);}
\foreach \i in {1,3,5}{
	\node[v1] at (A.corner \i) {};
	\node[v0] at (B.corner \i) {};
}
\foreach \i in {2,4,6}{
	\node[v0] at (A.corner \i) {};
	\node[v1] at (B.corner \i) {};
}
\end{scope}
\end{tikzpicture}
}
\caption{A $6$-vertex (top) and a $4$-vertex (bottom) insertion.}
\label{fig:8-12-vertices}
\end{figure}
In fact, similar insertions may be used to generate all primitives on 14, 16 , and 18 vertices as shown in Figure~\ref{fig:14vertices}, Figure~\ref{fig:16vertices}, and Figure~\ref{fig:18vertices}, respectively. Observe that to obtain all primitives on 18 vertices, we used both types of insertion.

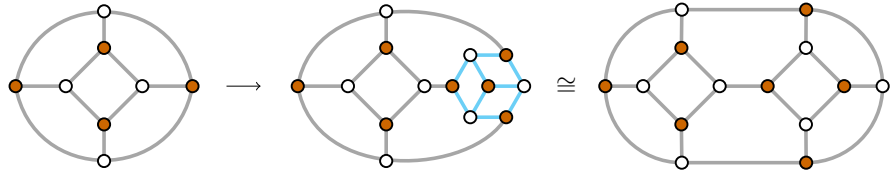
\begin{figure}[ht]
\scalebox{0.9}{
\begin{tikzpicture}
\begin{scope}
\node[regular polygon, regular polygon sides=4, minimum size=32, poly, rotate=45] at (0,0) (A) {};
\draw[poly] (-1.3,0) .. controls +(80:42pt) and +(100:42pt) .. (1.3,0);
\draw[poly] (-1.3,0) .. controls +(-80:42pt) and +(-100:42pt) .. (1.3,0);
\draw[poly] (A.corner 1) -- (0,1.1) node[v0]{};
\draw[poly] (A.corner 2) -- (-1.3,0) node[v1]{};
\draw[poly] (A.corner 3) -- (0,-1.1) node[v0]{};
\draw[poly] (A.corner 4) -- (1.3,0) node[v1]{};
\foreach \i in {1,3}{
	\node[v1] at (A.corner \i) {};
}
\foreach \i in {2,4}{
	\node[v0] at (A.corner \i) {};
}
\draw[->] (1.8,0) -- (2.3,0);
\end{scope}
\begin{scope}[xshift=118]
\node[regular polygon, regular polygon sides=4, minimum size=32, poly, rotate=45] at (0,0) (A) {};
\node[regular polygon, regular polygon sides=6, minimum size=30, poly,insertcol] at (1.5,0) (B) {};
\draw[poly] (-1.3,0) .. controls +(80:42pt) and +(120:28pt) .. (B.corner 1);
\draw[poly] (-1.3,0) .. controls +(-80:42pt) and +(-120:28pt) .. (B.corner 5);
\draw[poly] (A.corner 1) -- (0,1.1) node[v0]{};
\draw[poly] (A.corner 2) -- (-1.3,0) node[v1]{};
\draw[poly] (A.corner 3) -- (0,-1.1) node[v0]{};
\draw[poly] (A.corner 4) -- (B.corner 3);
\foreach \i in {1,3}{
	\node[v1] at (A.corner \i) {};
}
\foreach \i in {2,4}{
	\node[v0] at (A.corner \i) {};
}
\foreach \i in {2,4,6}{
	\draw[poly,insertcol] (1.5,0) -- (B.corner \i);
}
\node[v1] at (1.5,0){};
\foreach \i in {1,3,5}{
	\node[v1] at (B.corner \i) {};
}
\foreach \i in {2,4,6}{
	\node[v0] at (B.corner \i) {};
}
\node at (2.65,0) {$\cong$};
\end{scope}
\begin{scope}[xshift=215]
\node[shape=rounded rectangle, inner xsep=58, inner ysep=32, poly] at (1.85,0) (H) {};
\node[regular polygon, regular polygon sides=4, minimum size=32, rotate=45, poly] at (0.93,0) (A) {};
\node[regular polygon, regular polygon sides=4, minimum size=32, rotate=45, poly] at (2.76,0) (B) {};
\draw[poly] (A.corner 1) -- (H.north west);
\draw[poly] (A.corner 2) -- (H.west);
\draw[poly] (A.corner 3) -- (H.south west);
\draw[poly] (A.corner 4) -- (B.corner 2);
\draw[poly] (B.corner 1) -- (H.north east);
\draw[poly] (B.corner 4) -- (H.east);
\draw[poly] (H.south east) -- (B.corner 3);
\foreach \i in {1,3} \node[v1] at (A.corner \i) {};
\foreach \i in {2,4} \node[v0] at (A.corner \i) {};
\foreach \i in {1,3} \node[v0] at (B.corner \i) {};
\foreach \i in {2,4} \node[v1] at (B.corner \i) {};
\foreach \vertex in {east,north west,south west} \node[v0] at (H.\vertex) {};
\foreach \vertex in {north east,west,south east} \node[v1] at (H.\vertex) {};
\end{scope}
\end{tikzpicture}
}
\caption{Generating all primitives on $14$ vertices.}
\label{fig:14vertices}
\end{figure}

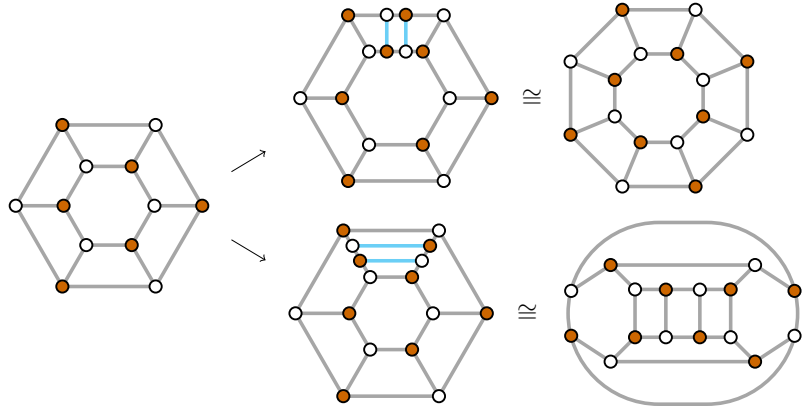
\begin{figure}[ht]
\scalebox{0.9}{
\begin{tikzpicture}
\begin{scope}
\node[regular polygon, regular polygon sides=6, minimum size=38, poly] at (0,0) (A) {};
\node[regular polygon, regular polygon sides=6, minimum size=78, poly] at (0,0) (B) {};
\foreach \i in {1,...,6} {\draw[poly] (A.corner \i) -- (B.corner \i);}
\foreach \i in {1,3,5}{
	\node[v1] at (A.corner \i) {};
	\node[v0] at (B.corner \i) {};
}
\foreach \i in {2,4,6}{
	\node[v0] at (A.corner \i) {};
	\node[v1] at (B.corner \i) {};
}
\draw[->] (1.8,0.5) -- (2.3,0.8);
\draw[->] (1.8,-0.5) -- (2.3,-0.8);
\end{scope}
\begin{scope}[xshift=120,yshift=45]
\node[regular polygon, regular polygon sides=6, minimum size=45, poly] at (0,0) (A) {};
\node[regular polygon, regular polygon sides=6, minimum size=80, poly] at (0,0) (B) {};
\foreach \i in {1,...,6} {\draw[poly] (A.corner \i) -- (B.corner \i);}
\foreach \i in {1,3,5}{
	\node[v1] at (A.corner \i) {};
	\node[v0] at (B.corner \i) {};
}
\foreach \i in {2,4,6}{
	\node[v0] at (A.corner \i) {};
	\node[v1] at (B.corner \i) {};
}
\draw[poly,insertcol] (A.corner 1)+(-0.25,0) node[v0]{} -- +(-0.25,0.54) node[v1]{};
\draw[poly,insertcol] (A.corner 1)+(-0.53,0) node[v1]{} -- +(-0.53,0.54) node[v0]{};
\node at (2,0) {$\cong$};
\end{scope}
\begin{scope}[xshift=230,yshift=45]
\node[regular polygon, regular polygon sides=8, minimum size=40, poly] at (0,0) (A) {};
\node[regular polygon, regular polygon sides=8, minimum size=80, poly] at (0,0) (B) {};
\foreach \i in {1,...,8} {\draw[poly] (A.corner \i) -- (B.corner \i);}
\foreach \i in {1,3,5,7}{
	\node[v1] at (A.corner \i) {};
	\node[v0] at (B.corner \i) {};
}
\foreach \i in {2,4,6,8}{
	\node[v0] at (A.corner \i) {};
	\node[v1] at (B.corner \i) {};
}
\end{scope}
\begin{scope}[xshift=118,yshift=-45]
\node[regular polygon, regular polygon sides=6, minimum size=35, poly] at (0,0) (A) {};
\node[regular polygon, regular polygon sides=6, minimum size=80, poly] at (0,0) (B) {};
\foreach \i in {1,...,6} {\draw[poly] (A.corner \i) -- (B.corner \i);}
\foreach \i in {1,3,5}{
	\node[v1] at (A.corner \i) {};
	\node[v0] at (B.corner \i) {};
}
\foreach \i in {2,4,6}{
	\node[v0] at (A.corner \i) {};
	\node[v1] at (B.corner \i) {};
}
\draw[poly,insertcol] (A.corner 2)+(-0.26,0.46) node[v0]{} -- +(0.88,0.46) node[v1]{};
\draw[poly,insertcol] (A.corner 2)+(-0.15,0.24) node[v1]{} -- +(0.76,0.24) node[v0]{};
\node at (2,0) {$\cong$};
\end{scope}
\begin{scope}[xshift=240,yshift=-45]
\node[shape=rounded rectangle, inner xsep=48, inner ysep=38, poly] at (0,0) (H) {};
\node[shape=rectangle, inner xsep=30, inner ysep=20] at (0,0) (B) {};
\node[shape=rectangle, inner xsep=20, inner ysep=10, poly] at (0,0) (A) {};
\draw[poly] (A.south)+(-0.25,0) node[v0]{} -- +(-0.25,0.7) node[v1]{}; 
\draw[poly] (A.south)+(0.25,0) node[v1]{} -- +(0.25,0.7) node[v0]{}; 
\draw[poly] (A.north west) -- (B.north west) -- +(-0.58,-0.38) node[v0]{};
\draw[poly] (A.south west) -- (B.south west) -- +(-0.58,0.38) node[v1]{};
\draw[poly] (A.north east) -- (B.north east) -- +(0.58,-0.38) node[v1]{};
\draw[poly] (A.south east) -- (B.south east) -- +(0.58,0.38) node[v0]{};
\draw[poly] (B.north east) -- (B.north west);
\draw[poly] (B.south east) -- (B.south west);
\foreach \vertex in {north west,south east} \node[v0] at (A.\vertex) {};
\foreach \vertex in {north east,south west} \node[v1] at (A.\vertex) {};
\foreach \vertex in {north east,south west} \node[v0] at (B.\vertex) {};
\foreach \vertex in {north west,south east} \node[v1] at (B.\vertex) {};
\end{scope}
\end{tikzpicture}
}
\caption{Generating all primitives on $16$ vertices.}
\label{fig:16vertices}
\end{figure}

\begin{figure}[ht]
\scalebox{0.85}{
\begin{tikzpicture}
\begin{scope}
\node[regular polygon, regular polygon sides=6, minimum size=38, poly] at (0,0) (A) {};
\node[regular polygon, regular polygon sides=6, minimum size=78, poly] at (0,0) (B) {};
\foreach \i in {1,...,6} {\draw[poly] (A.corner \i) -- (B.corner \i);}
\foreach \i in {1,3,5}{
	\node[v1] at (A.corner \i) {};
	\node[v0] at (B.corner \i) {};
}
\foreach \i in {2,4,6}{
	\node[v0] at (A.corner \i) {};
	\node[v1] at (B.corner \i) {};
}
\draw[->] (2,0) -- (2.5,0);
\end{scope}
\begin{scope}[xshift=128]
\node[regular polygon, regular polygon sides=6, minimum size=36, poly] at (-0.25,0) (A) {};
\node[regular polygon, regular polygon sides=6, minimum size=80] at (0,0) (B) {};
\node[regular polygon, regular polygon sides=6, minimum size=32, poly,insertcol] at (B.corner 6) (C) {};
\foreach \i/\j in {1/2,2/3,3/4,4/5} {
	\draw[poly] (B.corner \i) -- (B.corner \j);
}
\draw[poly] (B.corner 1) -- (C.corner 1);
\draw[poly] (B.corner 5) -- (C.corner 5);
\draw[poly] (A.corner 6) -- (C.corner 3);
\foreach \i in {2,4,6}{
	\draw[poly,insertcol] (B.corner 6) -- (C.corner \i);
}
\foreach \i in {1,3,5}{
	\node[v1] at (C.corner \i) {};
}
\foreach \i in {2,4,6}{
	\node[v0] at (C.corner \i) {};
}
\foreach \i in {1,...,5} {\draw[poly] (A.corner \i) -- (B.corner \i);}
\foreach \i in {1,3,5}{
	\node[v1] at (A.corner \i) {};
	\node[v0] at (B.corner \i) {};
}
\foreach \i in {2,4,6}{
	\node[v0] at (A.corner \i) {};
	\node[v1] at (B.corner \i) {};
}
\node at (2.55,0) {$\cong$};
\end{scope}
\begin{scope}[xshift=220]
\node[shape=rounded rectangle, inner xsep=60, inner ysep=35, poly] at (2,0) (H) {};
\node[regular polygon, regular polygon sides=6, minimum size=35, poly] at (1.2,0) (A) {};
\draw[poly] (H.west) node[v0]{} -- (A.corner 3); 
\draw[poly] (A.corner 1) -- +(0,18pt);
\draw[poly] (A.corner 2) -- +(0,18pt);
\draw[poly] (A.corner 4) -- +(0,-18pt);
\draw[poly] (A.corner 5) -- +(0,-18pt);
\node[above=17pt,v0] at (A.corner 1) {};
\node[above=17pt,v1] at (A.corner 2) {};
\node[below=17pt,v1] at (A.corner 4) {};
\node[below=17pt,v0] at (A.corner 5) {};
\node[regular polygon, regular polygon sides=4, minimum size=28, rotate=45, poly] at (3,0) (B) {};
\draw[poly] (A.corner 6) -- (B.corner 2); 
\draw[poly] (B.corner 4) -- (H.east) node[v0]{}; 
\draw[poly] (B.corner 1) -- +(0,20pt);
\draw[poly] (B.corner 3) -- +(0,-20pt);
\node[above=18pt,v1] at (B.corner 1) {};
\node[below=18pt,v1] at (B.corner 3) {};
\foreach \i in {2,4,6} \node[v0] at (A.corner \i) {};
\foreach \i in {1,3,5} \node[v1] at (A.corner \i) {};
\foreach \i in {1,3} \node[v0] at (B.corner \i) {};
\foreach \i in {2,4} \node[v1] at (B.corner \i) {};
\end{scope}
\end{tikzpicture}
}
\bigskip

\scalebox{0.85}{
\begin{tikzpicture}
\begin{scope}
\node[shape=rounded rectangle, inner xsep=60, inner ysep=35, poly] at (0,0) (H) {};
\node[regular polygon, regular polygon sides=4, minimum size=32, rotate=45, poly] at (-0.88,0) (A) {};
\node[regular polygon, regular polygon sides=4, minimum size=32, rotate=45, poly] at (0.88,0) (B) {};
\draw[poly] (A.corner 1) -- (H.north west);
\draw[poly] (A.corner 2) -- (H.west);
\draw[poly] (A.corner 3) -- (H.south west);
\draw[poly] (A.corner 4) -- (B.corner 2);
\draw[poly] (B.corner 1) -- (H.north east);
\draw[poly] (B.corner 4) -- (H.east);
\draw[poly] (H.south east) -- (B.corner 3);
\foreach \i in {1,3} \node[v0] at (A.corner \i) {};
\foreach \i in {2,4} \node[v1] at (A.corner \i) {};
\foreach \i in {1,3} \node[v1] at (B.corner \i) {};
\foreach \i in {2,4} \node[v0] at (B.corner \i) {};
\foreach \vertex in {east,north west,south west} \node[v1] at (H.\vertex) {};
\foreach \vertex in {north east,west,south east} \node[v0] at (H.\vertex) {};
\draw[->] (2.5,0) -- (3,0);
\end{scope}
\begin{scope}[xshift=156]
\node[shape=rounded rectangle, inner xsep=60, inner ysep=35, poly] at (0,0) (H) {};
\node[regular polygon, regular polygon sides=4, minimum size=32, rotate=45, poly] at (-0.88,-0.1) (A) {};
\node[regular polygon, regular polygon sides=4, minimum size=32, rotate=45, poly] at (0.88,-0.1) (B) {};
\draw[poly] (A.corner 1) -- (H.north west);
\draw[poly] (A.corner 2) -- (H.west);
\draw[poly] (A.corner 3) -- (H.south west);
\draw[poly] (A.corner 4) -- (B.corner 2);
\draw[poly] (B.corner 1) -- (H.north east);
\draw[poly] (B.corner 4) -- (H.east);
\draw[poly] (H.south east) -- (B.corner 3);
\foreach \i in {1,3} \node[v0] at (A.corner \i) {};
\foreach \i in {2,4} \node[v1] at (A.corner \i) {};
\foreach \i in {1,3} \node[v1] at (B.corner \i) {};
\foreach \i in {2,4} \node[v0] at (B.corner \i) {};
\foreach \vertex in {east,north west,south west} \node[v1] at (H.\vertex) {};
\foreach \vertex in {north east,west,south east} \node[v0] at (H.\vertex) {};
\draw[poly,insertcol] (A.corner 1)+(0,0.52) node[v0]{} -- +(1.762,0.52) node[v1]{};
\draw[poly,insertcol] (A.corner 1)+(0,0.26) node[v1]{} -- +(1.762,0.26) node[v0]{};
\node at (2.6,0) {$\cong$};
\end{scope}
\begin{scope}[xshift=288]
\node[poly, regular polygon, regular polygon sides=6, minimum size=45] at (0,0) (A) {};
\node[poly, regular polygon, regular polygon sides=6, minimum size=90] at (0,0) (B) {};
\node[poly, regular polygon, regular polygon sides=6, minimum size=70, white] at (0,0) (C) {};
\foreach \i in {1,...,6} \draw[poly] (A.corner \i) -- (B.corner \i);
\foreach \i/\j in {1/2,3/4,5/6} \draw[poly] (C.corner \i) -- (C.corner \j);
\foreach \i in {2,4,6}{
	\node[v1] at (A.corner \i) {};
	\node[v1] at (B.corner \i) {};
	\node[v0] at (C.corner \i) {};
}
\foreach \i in {1,3,5}{ 
	\node[v0] at (A.corner \i) {};
	\node[v0] at (B.corner \i) {};
	\node[v1] at (C.corner \i) {};
}
\end{scope}
\end{tikzpicture}
}
\caption{Generating all primitives on $18$ vertices.}
\label{fig:18vertices}
\end{figure}
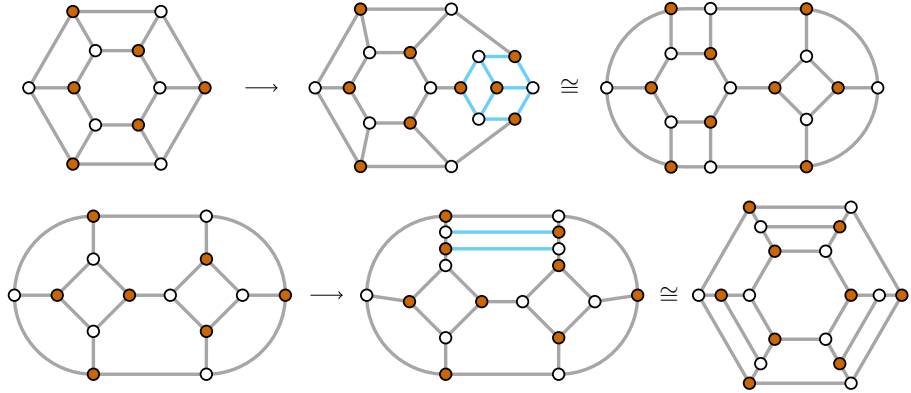

In this section, we will prove that for $n\ge 4$, {\em every} 3-connected bicubic planar map on $2n$ vertices can be constructed from a smaller primitive by means of the two insertion operations illustrated in the above examples. These operations were previously discussed in \cite{HKKN12} and referred to as 2-bridging (4-vertex addition) and hexagon addition (6-vertex addition).

First, let us formalize the needed insertion operations. Suppose a 3-connected bicubic planar map $M$ on $2n$ vertices is given.
\begin{enumerate}
\item[(a)] To insert 4 vertices into $M$, we select a face $F$ and choose two of its edges, say $e$ and $e'$, that are connected by an odd number of edges in both directions along the boundary of $F$. We then insert two additional vertices on each of $e$ and $e'$, preserving the bipartite property, and connect each vertex on $e$ with a vertex on $e'$ through $F$ such that the resulting map, call it $\hat M_{F''}$, is bicubic and planar (see Figure~\ref{fig:insertions}a). Here $F''$ stands for the face $F$ with the distinguished edges $e$ and $e'$. 
\item[(b)] To insert 6 vertices into $M$, we select a vertex $v$ and blow it up by adding a hexagon centered at $v$. We then connect every other vertex of the hexagon with the three vertices originally adjacent to $v$ and connect the remaining three vertices of the hexagon with vertex $v$, coloring the added vertices such that the resulting map, call it $\hat M_v$, is bicubic and planar (see Figure~\ref{fig:insertions}b).
\end{enumerate}

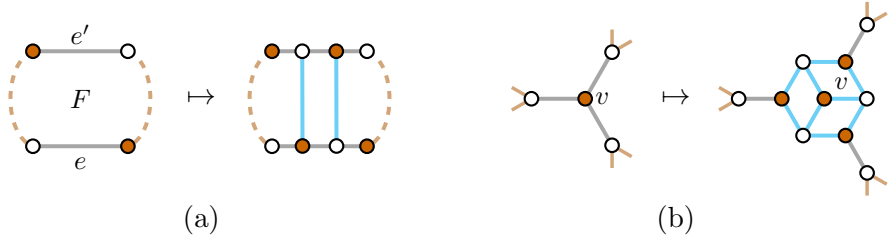
\begin{figure}[ht]
\begin{tikzpicture}
\begin{scope}
\node[regular polygon, regular polygon sides=4, minimum size=50] at (0,0) (A) {};
\draw[poly,dashed,brown!70] (A.corner 2) .. controls +(200:12pt) and +(160:12pt) .. (A.corner 3);
\draw[poly,dashed,brown!70] (A.corner 1) .. controls +(-20:12pt) and +(20:12pt) .. (A.corner 4);
\draw[poly] (A.corner 2) node[v1]{} -- (A.corner 1) node[v0]{}; 
\draw[poly] (A.corner 4) node[v1]{} -- (A.corner 3) node[v0]{}; 
\node at (0,0) {\small $F$};
\node at (0,0.85) {\small $e'$};
\node at (0,-0.85) {\small $e$};
\node at (1.6,0) {$\mapsto$};
\node at (1.6,-1.6) {(a)};
\end{scope}
\begin{scope}[xshift=90]
\node[regular polygon, regular polygon sides=4, minimum size=50] at (0,0) (A) {};
\draw[poly,dashed,brown!70] (A.corner 2) .. controls +(200:12pt) and +(160:12pt) .. (A.corner 3);
\draw[poly,dashed,brown!70] (A.corner 1) .. controls +(-20:12pt) and +(20:12pt) .. (A.corner 4);
\draw[poly] (A.corner 2) node[v1]{} -- (A.corner 1) node[v0]{}; 
\draw[poly] (A.corner 4) node[v1]{} -- (A.corner 3) node[v0]{};
\draw[poly,insertcol] (A.corner 2)+(0.4,0) node[v0]{} -- (-0.225,-0.625) node[v1]{};
\draw[poly,insertcol] (A.corner 2)+(0.85,0) node[v1]{} -- (0.225,-0.625) node[v0]{};
\end{scope}
\begin{scope}[xshift=190]
\node[regular polygon, regular polygon sides=6, minimum size=40] at (0,0) (A) {};
\draw[poly2] (A.corner 1) -- +(90:0.3);
\draw[poly2] (A.corner 1) -- +(30:0.3);
\draw[poly2] (A.corner 3) -- +(150:0.3);
\draw[poly2] (A.corner 3) -- +(210:0.3);
\draw[poly2] (A.corner 5) -- +(-30:0.3);
\draw[poly2] (A.corner 5) -- +(-90:0.3);
\foreach \i in {1,3,5}{ \draw[poly] (0,0) -- (A.corner \i) node[v0]{};}
\node[v1] at (0,0) {};
\node[right] at (0,0) {\small $v$};
\node at (1.2,0) {$\mapsto$};
\node at (1.2,-1.6) {(b)};
\end{scope}
\begin{scope}[xshift=280]
\node[regular polygon, regular polygon sides=6, minimum size=32,poly,insertcol] at (0,0) (A) {};
\node[regular polygon, regular polygon sides=6, minimum size=64] at (0,0) (B) {};
\draw[poly2] (B.corner 1) -- +(90:0.3);
\draw[poly2] (B.corner 1) -- +(30:0.3);
\draw[poly2] (B.corner 3) -- +(150:0.3);
\draw[poly2] (B.corner 3) -- +(210:0.3);
\draw[poly2] (B.corner 5) -- +(-30:0.3);
\draw[poly2] (B.corner 5) -- +(-90:0.3);
\foreach \i in {1,3,5}{ \draw[poly] (A.corner \i) -- (B.corner \i) node[v0]{};}
\foreach \i in {2,4,6}{ \draw[poly,insertcol] (0,0) -- (A.corner \i) node[v0]{};}
\foreach \i in {1,3,5}{ \node[v1] at (A.corner \i){};}
\node[v1] at (0,0) {};
\node[above right=0] at (0,0) {\small $v$};
\end{scope}
\end{tikzpicture}
\caption{The two vertex insertion operations.}
\label{fig:insertions}
\end{figure}

By construction, both $\hat M_{F''}$ and $\hat M_v$ are 3-connected bicubic planar maps on $2n+4$ and $2n+6$ vertices, respectively.

\begin{theorem}\label{thm:primitiveConstruction}
For $n\ge 4$, every 3-connected bicubic planar map on $2n$ vertices can be constructed from one on $2n-4$ or $2n-6$ vertices by means of the above two operations.
\end{theorem}

Before we go on to prove this result, let us illustrate the effect of the above two insertion operations on the dual planar maps. It is worth mentioning that if $M$ is cubic, then its dual map $M^*$ is a plane triangulation. Also, if $M$ is bipartite, then every vertex of $M^*$ has even degree, and the faces of $M^*$ can be 2-colored such that any pair of adjacent faces have different colors. In our discussion below, we will ignore the coloring in the dual.

\begin{figure}[ht]
\begin{tikzpicture}
\begin{scope}[yshift=84]
\node[regular polygon, regular polygon sides=4, minimum size=50] at (0,0) (A) {};
\draw[poly,dashed,label] (A.corner 2) .. controls +(200:12pt) and +(160:12pt) .. (A.corner 3);
\draw[poly,dashed,label] (A.corner 1) .. controls +(-20:12pt) and +(20:12pt) .. (A.corner 4);
\draw[poly] (A.corner 2) node[v1]{} -- (A.corner 1) node[v0]{}; 
\draw[poly] (A.corner 4) node[v1]{} -- (A.corner 3) node[v0]{};
\node at (0,0) {\small $F$};
\node at (0,0.85) {\small $e'$};
\node at (0,-0.8) {\small $e$};
\draw[thick,->] (0,-1.2) -- (0,-1.8);
\node[left=10pt] at (0,-1.4) {\small $M\mapsto \hat M_{F''}$}; 
\draw[thick,<->] (1.5,0) -- (2,0) node[above]{$*$} -- (2.5,0);
\end{scope}
\begin{scope}[xshift=115,yshift=84]
\draw[poly] (0,-0.7) node[v,black]{} -- (0,0.7) node[v,black]{};
\foreach \a in {-60,0,60,180}{\draw[poly,dashed,label] (0,0) -- (\a:0.6);} 
\node[v] at (0,0){}; 
\draw[thick,->] (0,-1.2) -- (0,-1.8);
\node[right=10pt] at (0,-1.4) {\small $M^*\mapsto \hat M^*_{F''}$}; 
\end{scope}
\begin{scope}
\node[regular polygon, regular polygon sides=4, minimum size=50] at (0,0) (A) {};
\draw[poly,dashed,label] (A.corner 2) .. controls +(200:12pt) and +(160:12pt) .. (A.corner 3);
\draw[poly,dashed,label] (A.corner 1) .. controls +(-20:12pt) and +(20:12pt) .. (A.corner 4);
\draw[poly] (A.corner 2) node[v1]{} -- (A.corner 1) node[v0]{}; 
\draw[poly] (A.corner 4) node[v1]{} -- (A.corner 3) node[v0]{};
\draw[poly] (A.corner 2)+(0.4,0) node[v0]{} -- (-0.225,-0.625) node[v1]{};
\draw[poly] (A.corner 2)+(0.85,0) node[v1]{} -- (0.225,-0.625) node[v0]{};
\draw[thick,<->] (1.5,0) -- (2,0) node[above]{$*$} -- (2.5,0);
\end{scope}
\begin{scope}[xshift=115]
\draw[poly] (0,-0.7) -- (-0.5,0) -- (0,0.7);
\draw[poly] (0,-0.7) -- (0.5,0) -- (0,0.7);
\foreach \a in {180}{\draw[poly,dashed,label] (-0.5,0) -- +(\a:0.6);} 
\foreach \a in {-60,0,60}{\draw[poly,dashed,label] (0.5,0) -- +(\a:0.6);} 
\draw[poly] (-0.5,0) node[v,black]{} -- (0.5,0) node[v,black]{};
\draw[poly] (0,-0.7) node[v,black]{} -- (0,0.7) node[v,black]{};
\node[v] at (0,0){}; 
\end{scope}
\end{tikzpicture}
\caption{The 4-vertex insertion and its dual.}
\label{fig:4vertexInsertion}
\end{figure}
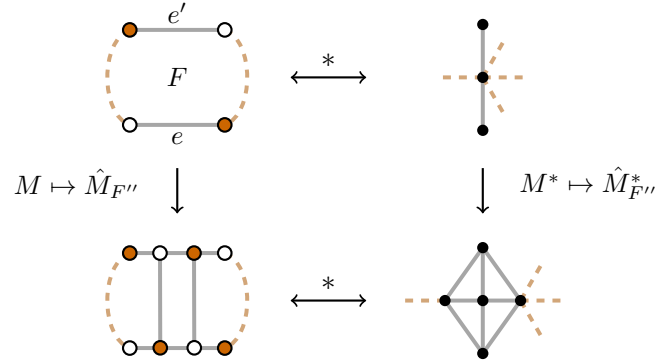

The first of our insertion operations is described by the diagram in Figure~\ref{fig:4vertexInsertion}, where the dashed edges in $M^*$ are supposed to represent the edges in the odd-length paths connecting $e$ and $e'$ in $M$. As expected, the planar map $\hat M^*_{F''}$ has four additional faces corresponding to the four additional vertices of $\hat M_{F''}$.

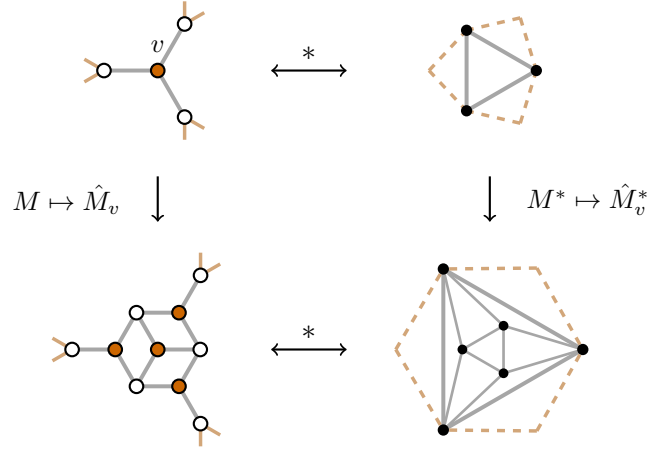
\begin{figure}[ht]
\begin{tikzpicture}
\tikzstyle{poly3}=[outer sep=0pt, draw, line width=1, gray!70]
\begin{scope}[yshift=105]
\node[regular polygon, regular polygon sides=6, minimum size=40] at (0,0) (A) {};
\foreach \i/\a in {1/30,1/90,3/150,3/210,5/-90,5/-30} {\draw[poly2] (A.corner \i) -- +(\a:0.3);}
\foreach \i in {1,3,5}{ \draw[poly] (0,0) -- (A.corner \i) node[v0]{};}
\node[v1] at (0,0) {};
\node[above=3pt] at (0,0) {\small $v$};
\draw[thick,->] (0,-1.4) -- (0,-2);
\node[left=10pt] at (0,-1.7) {\small $M\mapsto \hat M_{v}$}; 
\draw[thick,<->] (1.5,0) -- (2,0) node[above]{$*$} -- (2.5,0);
\end{scope}
\begin{scope}[xshift=125,yshift=105]
\node[regular polygon, regular polygon sides=3, minimum size=45,rotate=-30] at (0,0) (A) {};
\node[regular polygon, regular polygon sides=3, minimum size=35,poly,rotate=30] at (0,0) (B) {};
\foreach \i/\j in {1/1,1/3,2/1,2/2,3/2,3/3} {\draw[poly2,dashed] (A.corner \i) -- (B.corner \j);}  
\foreach \i in {1,2,3} {\node[v] at (B.corner \i) {};}
\draw[thick,->] (0,-1.4) -- (0,-2);
\node[right=10pt] at (0,-1.7) {\small $M^*\mapsto \hat M^*_{v}$}; 
\end{scope}
\begin{scope}
\node[regular polygon, regular polygon sides=6, minimum size=32,poly] at (0,0) (A) {};
\node[regular polygon, regular polygon sides=6, minimum size=64] at (0,0) (B) {};
\foreach \i/\a in {1/30,1/90,3/150,3/210,5/-90,5/-30} {\draw[poly2] (B.corner \i) -- +(\a:0.3);}
\foreach \i in {1,3,5}{ \draw[poly] (A.corner \i) -- (B.corner \i) node[v0]{};}
\foreach \i in {2,4,6}{ \draw[poly] (0,0) -- (A.corner \i) node[v0]{};}
\foreach \i in {1,3,5}{ \node[v1] at (A.corner \i){};}
\node[v1] at (0,0) {};
\draw[thick,<->] (1.5,0) -- (2,0) node[above]{$*$} -- (2.5,0);
\end{scope}
\begin{scope}[xshift=125]
\node[regular polygon, regular polygon sides=3, minimum size=70,rotate=-30] at (0,0) (A) {};
\node[regular polygon, regular polygon sides=3, minimum size=70,poly,rotate=30] at (0,0) (B) {};
\node[regular polygon, regular polygon sides=3, minimum size=20,poly3,rotate=-30] at (0,0) (C) {};
\foreach \i/\j in {1/1,1/3,2/1,2/2,3/2,3/3} {\draw[poly2,dashed] (A.corner \i) -- (B.corner \j);}  
\foreach \i/\j in {1/1,1/2,2/2,2/3,3/3,3/1} {\draw[poly3] (B.corner \i) -- (C.corner \j);}
\foreach \i in {1,2,3} {\node[v] at (B.corner \i) {};}
\foreach \i in {1,2,3} {\node[circle, inner sep=1.3, fill] at (C.corner \i) {};}
\end{scope}
\end{tikzpicture}
\caption{The 6-vertex insertion and its dual.}
\label{fig:6vertexInsertion}
\end{figure}

Our second insertion operation leads to the diagram in Figure~\ref{fig:6vertexInsertion}. Note that, as expected, the dual $ \hat M^*_{v}$ has six additional faces corresponding to the six additional vertices of $\hat M_{v}$.

\begin{proof}[Proof of Theorem~\ref{thm:primitiveConstruction}]
For $n\ge 4$, any 3-connected bicubic planar map $M$ on $2n$ vertices must contain a 4-face, say $F$, with incident faces $A$, $B$, $C$, and $D$. In the dual $M^*$, face $F$ together with its incident faces corresponds to a subgraph as the one shown in Figure~\ref{fig:primeSubgraph}(a). 
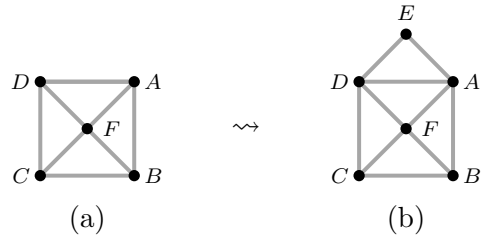
\begin{figure}[ht]
\begin{tikzpicture}
\begin{scope}
\node[regular polygon, regular polygon sides=4, minimum size=50, poly] at (0,0) (A) {};
\draw[poly] (A.corner 3) -- (A.corner 1) -- (A.corner 2) -- (A.corner 4);
\foreach \i in {1,2,3,4} {\node[v] at (A.corner \i){};}
\node[v] at (0,0){} node[right=2pt] {\scriptsize $F$};
\foreach \i/\a in {1/A,4/B} {\node[right] at (A.corner \i){\scriptsize $\a$};}
\foreach \i/\a in {2/D,3/C} {\node[left] at (A.corner \i){\scriptsize $\a$};}
\node at (0,-1.2) {(a)};
\end{scope}
\node at (2.1,0) {$\leadsto$};
\begin{scope}[xshift=120]
\node[regular polygon, regular polygon sides=4, minimum size=50, poly] at (0,0) (A) {};
\draw[poly] (A.corner 3) -- (A.corner 1) -- (0,1.25) node[v,black]{} -- (A.corner 2) -- (A.corner 4);
\node[above=3pt] at (0,1.2) {\scriptsize $E$};
\foreach \i in {1,2,3,4} {\node[v] at (A.corner \i){};}
\node[v] at (0,0){} node[right=2pt] {\scriptsize $F$};
\foreach \i/\a in {1/A,4/B} {\node[right] at (A.corner \i){\scriptsize $\a$};}
\foreach \i/\a in {2/D,3/C} {\node[left] at (A.corner \i){\scriptsize $\a$};}
\node at (0,-1.2) {(b)};
\end{scope}
\end{tikzpicture}
\caption{Subgraphs of $M^*$.}
\label{fig:primeSubgraph}
\end{figure} 
Moreover, since $M^*$ is a triangulation and its vertices have even degree, there must exist at least one additional vertex $E$ of $M^*$ that is part of a $3$-face incident to one of the edges of that dual subgraph. Thus, without loss of generality, $M^*$ must contain a subgraph as the one shown in Figure~\ref{fig:primeSubgraph}(b). 

Now, there are two possibilities: $(i)$ Vertices $A$, $B$, $E$ are on the same 3-face, and vertices $C$, $D$, $E$ are the vertices of another 3-face, see Figure~\ref{fig:ExtendedSubgraph}(a); or $(ii)$ one of these two sets of vertices, say the first one, does not define a 3-face, see Figure~\ref{fig:ExtendedSubgraph}(b), implying that vertex $A$ must have additional incident edges.

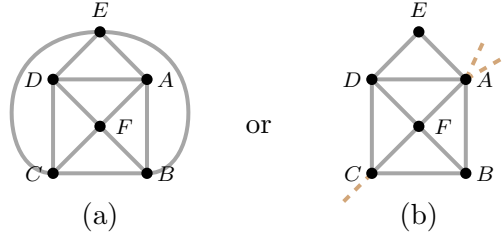
\begin{figure}[ht]
\begin{tikzpicture}
\begin{scope}
\node[regular polygon, regular polygon sides=4, minimum size=50, poly] at (0,0) (A) {};
\draw[poly] (0,1.25) to[out=0, in=100] (1.15,0.4) to[out=-80, in=0](A.corner 4);
\draw[poly] (0,1.25) to[out=180, in=80] (-1.15,0.4) to[out=-100, in=180](A.corner 3);
\draw[poly] (A.corner 3) -- (A.corner 1) -- (0,1.25) node[v,black]{} -- (A.corner 2) -- (A.corner 4);
\node[above=3pt] at (0,1.2) {\scriptsize $E$};
\foreach \i in {1,2,3,4} {\node[v] at (A.corner \i){};}
\node[v] at (0,0){} node[right=2pt] {\scriptsize $F$};
\foreach \i/\a in {1/A,4/B} {\node[right] at (A.corner \i){\scriptsize $\a$};}
\foreach \i/\a in {2/D,3/C} {\node[left] at (A.corner \i){\scriptsize $\a$};}
\node at (0,-1.2) {(a)};
\end{scope}
\node at (2.1,0) {or};
\begin{scope}[xshift=120]
\node[regular polygon, regular polygon sides=4, minimum size=50, poly] at (0,0) (A) {};
\draw[poly] (A.corner 3) -- (A.corner 1) -- (0,1.25) node[v,black]{} -- (A.corner 2) -- (A.corner 4);
\foreach \a in {30,65}{\draw[poly,dashed,label] (A.corner 1) -- +(\a:0.6);} 
\draw[poly,dashed,label] (A.corner 3) -- +(225:0.6);
\node[above=3pt] at (0,1.2) {\scriptsize $E$};
\foreach \i in {1,2,3,4} {\node[v] at (A.corner \i){};}
\node[v] at (0,0){} node[right=2pt] {\scriptsize $F$};
\foreach \i/\a in {1/A,4/B} {\node[right] at (A.corner \i){\scriptsize $\a$};}
\foreach \i/\a in {2/D,3/C} {\node[left] at (A.corner \i){\scriptsize $\a$};}
\node at (0,-1.2) {(b)};
\end{scope}
\end{tikzpicture}
\caption{Possible extended subgraphs of $M^*$.}
\label{fig:ExtendedSubgraph}
\end{figure} 

In case $(i)$, we recognize the dual of a 6-vertex insertion, so removing vertices $A$, $D$, $F$ and their incident edges from $M^*$ will give us the dual of a 3-connected bicubic planar map on $2n-6$ vertices. On the other hand, the subgraph in case $(ii)$ contains the dual of a 4-vertex insertion, so collapsing vertices $A$, $F$, $C$ in $M^*$ into a single vertex will give us the dual of a 3-connected bicubic planar map on $2n-4$ vertices.
\end{proof}

\begin{remark}
To construct all 3-connected bicubic planar maps, we need {\em both} of our insertion operations. For example, for $n>2$, the family of $2n$-prism graphs can all be built from the cube by means of iterative 4-vertex insertions, and they cannot be obtained from a smaller primitive via a 6-vertex insertion.

On the other hand, for $n\ge 1$, Figure~\ref{fig:infinite6vertex} shows a family of primitives on $8+6n$ vertices, built from the cube by means of iterative 6-vertex insertions, that cannot be obtained from a smaller primitive via a 4-vertex insertion.

\begin{figure}[ht]
\scalebox{0.8}{
\begin{tikzpicture}
\begin{scope}
\node[shape=rounded rectangle, inner xsep=56, inner ysep=33, poly] at (0,0) (H) {};
\node[regular polygon, regular polygon sides=4, minimum size=32, rotate=45, poly] at (-0.85,0) (A) {};
\node[regular polygon, regular polygon sides=4, minimum size=32, rotate=45, poly] at (0.87,0) (B) {};
\draw[poly] (A.corner 1) -- +(90:0.6) node[v1]{};
\draw[poly] (A.corner 3) -- +(-90:0.6) node[v1]{};
\draw[poly] (A.corner 2) -- (H.west) node[v0]{};
\draw[poly] (A.corner 4) -- (B.corner 2);
\draw[poly] (B.corner 1) -- +(90:0.6) node[v0]{};
\draw[poly] (B.corner 3) -- +(-90:0.6) node[v0]{};
\draw[poly] (B.corner 4) -- (H.east) node[v1]{};
\foreach \i in {1,3} \node[v0] at (A.corner \i) {};
\foreach \i in {2,4} \node[v1] at (A.corner \i) {};
\foreach \i in {1,3} \node[v1] at (B.corner \i) {};
\foreach \i in {2,4} \node[v0] at (B.corner \i) {};
\node[right=2pt] at (H.east) {,};
\end{scope}
\begin{scope}[xshift=156]
\node[shape=rounded rectangle, inner xsep=78, inner ysep=33, poly] at (0,0) (H) {};
\node[regular polygon, regular polygon sides=4, minimum size=32, rotate=45, poly] at (-1.67,0) (A) {};
\node[regular polygon, regular polygon sides=4, minimum size=32, rotate=45, poly] at (0,0) (B) {};
\node[regular polygon, regular polygon sides=4, minimum size=32, rotate=45, poly] at (1.67,0) (C) {};
\draw[poly] (A.corner 1) -- +(90:0.6) node[v1]{};
\draw[poly] (A.corner 3) -- +(-90:0.6) node[v1]{};
\draw[poly] (A.corner 2) -- (H.west) node[v0]{};
\draw[poly] (A.corner 4) -- (B.corner 2);
\draw[poly] (B.corner 1) -- +(90:0.6) node[v0]{};
\draw[poly] (B.corner 3) -- +(-90:0.6) node[v0]{};
\draw[poly] (B.corner 4) -- (C.corner 2);
\draw[poly] (C.corner 1) -- +(90:0.6) node[v1]{};
\draw[poly] (C.corner 3) -- +(-90:0.6) node[v1]{};
\draw[poly] (C.corner 4) -- (H.east) node[v0]{};
\foreach \i in {1,3} \node[v0] at (A.corner \i) {};
\foreach \i in {2,4} \node[v1] at (A.corner \i) {};
\foreach \i in {1,3} \node[v1] at (B.corner \i) {};
\foreach \i in {2,4} \node[v0] at (B.corner \i) {};
\foreach \i in {1,3} \node[v0] at (C.corner \i) {};
\foreach \i in {2,4} \node[v1] at (C.corner \i) {};
\node[right=2pt] at (H.east) {, $\dots$};
\end{scope}
\end{tikzpicture}
}
\caption{An infinite family of primitives generated by 6-vertex insertions.}
\label{fig:infinite6vertex}
\end{figure}
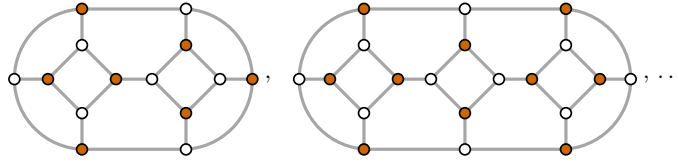

It is also worth mentioning that some primitive maps can be built with different insertion operations. For example, the map in Figure~\ref{fig:primitive18} can be obtained from the 6-prism by a 6-vertex insertion, or from the primitive map on 14 vertices by a 4-vertex insertion.

\begin{figure}[ht]
\begin{tikzpicture}[baseline=4]
\node[shape=rounded rectangle, inner xsep=60, inner ysep=32, poly] at (2,0) (H) {};
\node[regular polygon, regular polygon sides=6, minimum size=35, poly] at (1.2,0) (A) {};
\draw[poly] (H.west) node[v0]{} -- (A.corner 3); 
\draw[poly] (A.corner 1) -- +(0,17pt) node[v0]{};
\draw[poly] (A.corner 2) -- +(0,17pt) node[v1]{};
\draw[poly] (A.corner 4) -- +(0,-17pt) node[v1]{};
\draw[poly] (A.corner 5) -- +(0,-17pt) node[v0]{};
\node[regular polygon, regular polygon sides=4, minimum size=28, rotate=45, poly] at (3,0) (B) {};
\draw[poly] (A.corner 6) -- (B.corner 2); 
\draw[poly] (B.corner 4) -- (H.east) node[v0]{}; 
\draw[poly] (B.corner 1) -- +(0,18pt) node[v1]{};
\draw[poly] (B.corner 3) -- +(0,-18pt) node[v1]{};
\foreach \i in {2,4,6} \node[v0] at (A.corner \i) {};
\foreach \i in {1,3,5} \node[v1] at (A.corner \i) {};
\foreach \i in {1,3} \node[v0] at (B.corner \i) {};
\foreach \i in {2,4} \node[v1] at (B.corner \i) {};
\end{tikzpicture}
\caption{A primitive map on 18 vertices.}
\label{fig:primitive18}
\end{figure}
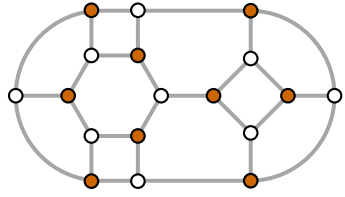
\end{remark}

\begin{remark}
The statement of Theorem~\ref{thm:primitiveConstruction} appears in \cite[Proposition~4]{HKKN12} as a consequence of work by Batagelj \cite{Bata89} on 3-connected quadrangulations. The link is somewhat unclear to us, so we chose to include a direct proof here.
\end{remark}

\subsection{Minimal and maximal number of rootings}
So far in this section, we have only discussed the construction of 3-connected bicubic planar maps, but we have not addressed the number of possible rootings. In theory, a primitive planar map on $2n$ vertices could have between 1 and $6n$ different rootings, depending on its symmetric properties. For example, the planar map $\Theta$ and the cube are both fully symmetric and have only one rooting each. On the other hand, the primitive map on 18 vertices shown in Figure~\ref{fig:primitive18} has 54 different rootings, meaning that every edge leads to two different rooted maps. This is the smallest primitive map having this maximality property.

In general, for $n>4$, it is easy to see that a 3-connected bicubic planar map on $2n$ vertices must have at least three rootings. But we can say more:

\begin{proposition} \label{prop:3rootings}
If a 3-connected bicubic planar map $M$ on $2n$ vertices has exactly three distinct rootings, then $M$ is either an $n$-prism with $n$ even, or $M$ is the truncated octahedron on 24 vertices; see Figure~\ref{fig:truncatedOcta}.
\end{proposition}
\begin{proof}
Any 3-connected bicubic planar map $M$ with exactly three distinct rootings, must have a 4-face and at least one $k$-face with $k\ge 6$, $k$ even.\footnote{The only primitive map with nothing but 4-faces is the cube, and it has only one rooting.} Any graph with a vertex $v$ incident to a 4-face and a $k$-face will have at least three distinct rootings:
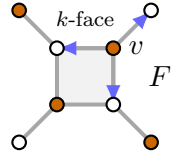
\begin{figure}[h!]
\begin{tikzpicture}
\node[regular polygon, regular polygon sides=4, minimum size=30, poly, fill=gray!10] at (0,0) (A) {};
\draw[root] (A.corner 1) -- (A.corner 2);
\draw[root] (A.corner 1) -- (A.corner 4);
\draw[root] (A.corner 1) -- +(0.5,0.5) node[v0]{};
\draw[poly] (A.corner 2) -- +(-0.5,0.5) node[v1]{};
\draw[poly] (A.corner 3) -- +(-0.5,-0.5) node[v0]{};
\draw[poly] (A.corner 4) -- +(0.5,-0.5) node[v1]{};
\foreach \i in {1,3}{ \node[v1] at (A.corner \i) {};}
\foreach \i in {2,4}{ \node[v0] at (A.corner \i) {};}
\node[right=2pt] at (A.corner 1) {$v$};
\node[above=15pt] at (0,0) {\scriptsize $k$-face};
\node[right=20pt] at (0,0) {$F$};
\end{tikzpicture}
\caption{Local picture around vertex $v$.}
\end{figure}

If $M$ has exactly three rootings, every vertex has to be of the same type as $v$, incident to a 4-face, to a $k$-face, and to a third face $F$ that must be either a 4-face or a $k$-face. Suppose $M$ has $2n$ vertices, $a$ 4-faces, and $b$ $k$-faces. Then we have $a+b = n+2$ and $4a +kb = 6n$.

If $F$ is a $k$-face, then every 4-face is surrounded by $k$-faces, hence $4a = 2n$ and $kb=4n$. Moreover, $a+b = n+2$ implies $2a+2b = 4a+4$, and so $b = a+2$. Therefore,
\[ k = \frac{4n}{b} = \frac{8a}{a+2} < 8 \quad \leadsto\quad k=6. \]
This further implies $a=6$, $b=8$, and $n=12$. It can be easily shown that $M$ must be the truncated octahedron on 24 vertices, see Figure~\ref{fig:truncatedOcta}.

On the other hand, if $F$ is a 4-face, then each $k$-face must be surrounded by a belt of 4-faces. This implies $b=2$, and so $a=n$ and $k=n$. This means that $M$ is the $n$-prism.
\end{proof}

\begin{figure}[ht]
\scalebox{0.85}{
\begin{tikzpicture}
\node[regular polygon, regular polygon sides=6, minimum size=32, poly] at (0,0) (A) {};
\node[regular polygon, regular polygon sides=12, minimum size=70, poly] at (0,0) (B) {};
\node[regular polygon, regular polygon sides=12, minimum size=108] at (0,0) (C) {};
\foreach \i/\j in {2/3,3/4,4/7,5/8,6/11,7/12} { \draw[poly] (A.corner \i) -- (B.corner \j);}
\foreach \i/\j in {1/2,5/6,9/10} { \draw[poly] (C.corner \i) -- (C.corner \j);}
\foreach \i in {1,2,5,6,9,10} { \draw[poly] (B.corner \i) -- (C.corner \i);}
\draw[poly] (C.corner 2) .. controls +(190:28pt) and +(110:28pt) .. (C.corner 5);
\draw[poly] (C.corner 6) .. controls +(-50:28pt) and +(230:28pt) .. (C.corner 9);
\draw[poly] (C.corner 10) .. controls +(70:28pt) and +(-10:28pt) .. (C.corner 1);
\foreach \i in {1,3,5}{ \node[v1] at (A.corner \i) {};}
\foreach \i in {2,4,6}{ \node[v0] at (A.corner \i) {};}
\foreach \i in {1,3,5,7,9,11}{ \node[v1] at (B.corner \i) {};}
\foreach \i in {2,4,6,8,10,12}{ \node[v0] at (B.corner \i) {};}
\foreach \i in {2,6,10}{ \node[v1] at (C.corner \i) {};}
\foreach \i in {1,5,9}{ \node[v0] at (C.corner \i) {};}
\end{tikzpicture}
}
\caption{Truncated octahedral graph}
\label{fig:truncatedOcta}
\end{figure}
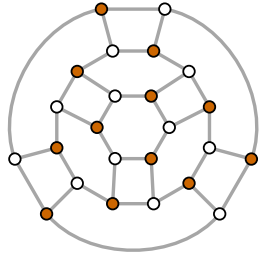

Note that if a 3-connected bicubic planar map on $2n$ vertices is an $n$-prism, then $n$ must be even. Thus the following corollary is a direct consequence of Proposition~\ref{prop:3rootings}.

\begin{corollary}
If $n\ge 7$ is odd, then every 3-connected bicubic planar map on $2n$ vertices must have at least four distinct rootings. 
\end{corollary}

\begin{proposition}[Asymmetric primitives]
For $n=9$ and for every $n\ge 11$, there exists a 3-connected bicubic planar map on $2n$ vertices with the maximum of $6n$ distinct rootings, obtained by either a 4-vertex or a 6-vertex insertion into an appropriate prism graph.
\end{proposition}
\begin{proof}
Since the number of rootings of a planar map $M$ with $3n$ edges equals $6n$ divided by the number of orientation preserving automorphism on $M$, the maximal number of rootings occurs precisely when there is only one such automorphism.

It was already mentioned that the map $M$ shown in Figure~\ref{fig:primitive18} can be obtained through a 6-vertex insertion into the $6$-prism graph. Note that $M$ has a single vertex incident to three 4-faces, and only one edge connecting that vertex to its unique 8-face. Clearly, the identity is the only orientation preserving automorphism on $M$, so it has $54$ ($=6n$) rootings.

A similar insertion into the $8$-prism graph generates a 3-connected bicubic planar map on 22 vertices ($n=11$), and in general, a 6-vertex insertion into a $2m$-prism generates a primitive map on $4m+6$ vertices ($n=2m+3$). With this process, we get a family of primitive maps on $2n$ vertices for $n=9, 11, 13,\dots$, having $6n$ distinct rootings.   

Now, we start with the $10$-prism $P$ and perform a 4-vertex insertion such that the inner 10-face of $P$ is subdivided into an 8-face, a 4-face, and a 6-face. Let $\hat P$ denote the resulting primitive map on 24 vertices ($n=12$). Note that $F_1$ is the only 8-face of $\hat P$, and $F_2$ is the only 4-face not adjacent to another 4-face in $\hat P$ (see Figure~\ref{fig:primitive24}). In other words, the edge incident to both $F_1$ and $F_2$ is unique, and the identity is the only orientation preserving automorphism on $\hat P$. Thus, $\hat P$ has 72 ($=6n$) rootings.

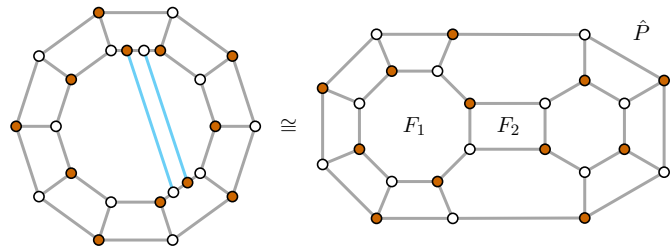
\begin{figure}[ht]
\scalebox{0.75}{
\begin{tikzpicture}
\begin{scope}
\node[regular polygon, regular polygon sides=10, minimum size=80, poly] at (0,0) (A) {};
\node[regular polygon, regular polygon sides=10, minimum size=120, poly] at (0,0) (B) {};
\foreach \i in {1,...,10} { \draw[poly] (A.corner \i) -- (B.corner \i);}
\draw[poly,insertcol] (A.corner 2)+(0.28,0) node[v1]{} -- +(1.10,-2.5) node[v0]{};
\draw[poly,insertcol] (A.corner 2)+(0.58,0) node[v0]{} -- +(1.35,-2.33) node[v1]{};
\foreach \i in {1,3,5,7,9}{ \node[v1] at (A.corner \i) {};}
\foreach \i in {2,4,6,8,10}{ \node[v0] at (A.corner \i) {};}
\foreach \i in {1,3,5,7,9}{ \node[v0] at (B.corner \i) {};}
\foreach \i in {2,4,6,8,10}{ \node[v1] at (B.corner \i) {};}
\node at (2.7,0) {$\cong$};
\end{scope}
\begin{scope}[xshift=140]
\node[regular polygon, regular polygon sides=8, minimum size=100, poly, white] at (0,0) (C) {};
\node[regular polygon, regular polygon sides=6, minimum size=92, rotate=30, poly, white] at (3,0) (D) {};
\node[regular polygon, regular polygon sides=8, minimum size=60, poly] at (0,0) (A) {};
\node[regular polygon, regular polygon sides=6, minimum size=46, rotate=30, poly] at (3,0) (B) {};
\draw[poly] (C.corner 1) -- (C.corner 2) -- (C.corner 3) -- (C.corner 4) -- (C.corner 5) -- (C.corner 6);
\draw[poly] (D.corner 4) -- (D.corner 5) -- (D.corner 6) -- (D.corner 1);
\draw[poly] (A.corner 7) -- (B.corner 3);
\draw[poly] (A.corner 8) -- (B.corner 2);
\draw[poly] (C.corner 1) -- (D.corner 1);
\draw[poly] (C.corner 6) -- (D.corner 4);
\foreach \i in {1,...,6} \draw[poly] (A.corner \i) -- (C.corner \i);
\foreach \i in {1,4,5,6} \draw[poly] (B.corner \i) -- (D.corner \i);
\foreach \i in {2,4,6,8} \node[v1] at (A.corner \i) {};
\foreach \i in {1,3,5,7} \node[v0] at (A.corner \i) {};
\foreach \i in {1,3,5} \node[v1] at (B.corner \i) {};
\foreach \i in {2,4,6} \node[v0] at (B.corner \i) {};
\foreach \i in {2,4,6} \node[v0] at (C.corner \i) {};
\foreach \i in {1,3,5} \node[v1] at (C.corner \i) {};
\foreach \i in {1,5} \node[v0] at (D.corner \i) {};
\foreach \i in {4,6} \node[v1] at (D.corner \i) {};
\node at (0,0) {$F_1$};
\node[left=0pt] at (2,0) {$F_2$};
\node at (4,1.7) {$\hat P$};
\end{scope}
\end{tikzpicture}
}
\caption{The primitive map $\hat P$.}
\label{fig:primitive24}
\end{figure}

In general, a 4-vertex insertion into a $2m$-prism ($m\ge 5$) in such a way that the inner $2m$-face is subdivided into a $(2m-2)$-face, a 4-face, and a 6-face, gives a primitive map on $4m+4$ vertices ($n=2m+2$) with the identity as the only orientation preserving automorphism. With this process, we can generate a family of primitive maps on $2n$ vertices for $n=12, 14, 16,\dots$, having $6n$ distinct rootings.   
\end{proof}

\section{Appendix}

The goal of this appendix is to further illustrate our construction from Section~\ref{sec:construction}.

We start with $n=3$ and build the library of all 12 rooted bicubic planar maps on $6$ ($=2n$) vertices using decorated Dyck paths of semilength 9 ($=3n$), see Table~\ref{tab:library9edges} and Figure~\ref{fig:library9edges}. In this case, we only need Dyck paths of the form $\U^3\D^i\U^3\D^j\U^3\D^k$ with $i,j,k\ge 1$, with ascents decorated by the primitive map $\Theta$. Note that the example given in the proof of Theorem~\ref{thm:mainBijection} used the Dyck path $\U^3\D^2\U^3\D\U^3\D^6$ to produce the planar map $(\Theta\merge_{\bf 2}\,\Theta)\merge_{\bf 4}\,\Theta$.

\begin{table}[ht]
\def\R{\rule[-1ex]{0ex}{3.6ex}}
\begin{tabular}{l|c}
\;\; Dyck path & Rooted bicubic map \\[1pt] \hline
\R $\U^3\D\U^3\D\U^3\D^7$ &  $(\Theta\merge_{\bf 1}\,\Theta)\merge_{\bf 4}\,\Theta$ \\
$\U^3\D\U^3\D^2\U^3\D^6$ &  $(\Theta\merge_{\bf 1}\,\Theta)\merge_{\bf 5}\,\Theta$ \\
$\U^3\D\U^3\D^3\U^3\D^5$ & $(\Theta\merge_{\bf 1}\,\Theta)\merge_{\bf 6}\,\Theta$ \\
$\U^3\D\U^3\D^4\U^3\D^4$ & $(\Theta\merge_{\bf 1}\,\Theta)\merge_{\bf 2}\,\Theta$ \\
$\U^3\D\U^3\D^5\U^3\D^3$ & $(\Theta\merge_{\bf 1}\,\Theta)\merge_{\bf 3}\,\Theta$ \\
$\U^3\D^2\U^3\D\U^3\D^6$ & $(\Theta\merge_{\bf 2}\,\Theta)\merge_{\bf 4}\,\Theta$ \\
$\U^3\D^2\U^3\D^2\U^3\D^5$ & $(\Theta\merge_{\bf 2}\,\Theta)\merge_{\bf 5}\,\Theta$ \\
$\U^3\D^2\U^3\D^3\U^3\D^4$ & $(\Theta\merge_{\bf 2}\,\Theta)\merge_{\bf 6}\,\Theta$ \\
$\U^3\D^2\U^3\D^4\U^3\D^3$ & $(\Theta\merge_{\bf 2}\,\Theta)\merge_{\bf 3}\,\Theta$ \\
$\U^3\D^3\U^3\D\U^3\D^5$ & $(\Theta\merge_{\bf 3}\,\Theta)\merge_{\bf 4}\,\Theta$ \\
$\U^3\D^3\U^3\D^2\U^3\D^4$ & $(\Theta\merge_{\bf 3}\,\Theta)\merge_{\bf 5}\,\Theta$ \\
$\U^3\D^3\U^3\D^3\U^3\D^3$ & $(\Theta\merge_{\bf 3}\,\Theta)\merge_{\bf 6}\,\Theta$
\end{tabular}
\bigskip
\caption{Construction of rooted bicubic planar maps on $6$ vertices.}
\label{tab:library9edges}
\end{table}

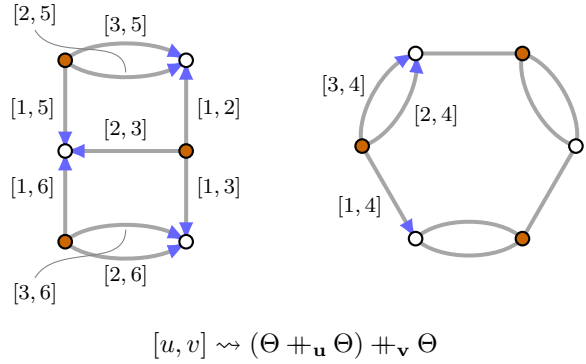
\begin{figure}[ht]
\begin{tikzpicture}[scale=0.8]
\begin{scope}
\draw[root] (0,0) .. controls +(45:15pt) and +(135:15pt) .. (2,0);
\draw[root] (0,0) .. controls +(-45:15pt) and +(-135:15pt) .. (2,0);
\draw[root] (2,1.5) -- (0,1.5);
\draw[root] (0,3) .. controls +(45:15pt) and +(135:15pt) .. (2,3);
\draw[root] (0,3) .. controls +(-45:15pt) and +(-135:15pt) .. (2,3);
\foreach \y in {0,3} {\draw[root] (0,\y) -- (0,1.5);}
\foreach \y in {0,3} {\draw[root] (2,1.5) -- (2,\y);}
\node[v1] at (0,0) {}; \node[v0] at (2,0) {};
\node[v0] at (0,1.5) {}; \node[v1] at (2,1.5) {};
\node[v1] at (0,3) {}; \node[v0] at (2,3) {};
\node[above=5] at (1,3) {\scriptsize $[3,5]$};
\node[above] at (1,1.5) {\scriptsize $[2,3]$};
\node[below=5] at (1,0) {\scriptsize $[2,6]$};
\node[left] at (0,2.2) {\scriptsize $[1,5]$};
\node[left] at (0,0.9) {\scriptsize $[1,6]$};
\node[right] at (2,2.2) {\scriptsize $[1,2]$};
\node[right] at (2,0.9) {\scriptsize $[1,3]$};
\draw[gray] (1,2.8) to[out=90, in=-60] (-0.5,3.5) node[above=-2,black] {\scriptsize $[2,5]$};
\draw[gray] (1,0.2) to[out=-90, in=60] (-0.5,-0.6) node[below=-2,black] {\scriptsize $[3,6]$};
\end{scope}
\begin{scope}[xshift=190,yshift=45]
\node[regular polygon, regular polygon sides=6, minimum size=80] at (0,0) (A) {};
\draw[root] (A.corner 3) .. controls +(105:15pt) and +(195:15pt) .. (A.corner 2);
\draw[root] (A.corner 3) .. controls +(15:15pt) and +(-75:15pt) .. (A.corner 2);
\draw[root] (A.corner 3) -- (A.corner 4);
\draw[poly] (A.corner 4) .. controls +(45:15pt) and +(135:15pt) .. (A.corner 5);
\draw[poly] (A.corner 4) .. controls +(-45:15pt) and +(-135:15pt) .. (A.corner 5);
\draw[poly] (A.corner 6) .. controls +(75:15pt) and +(-15:15pt) .. (A.corner 1);
\draw[poly] (A.corner 6) .. controls +(165:15pt) and +(-105:15pt) .. (A.corner 1);
\foreach \i/\j in {1/2,5/6}{
	\draw[poly] (A.corner \i) -- (A.corner \j);
}
\foreach \i in {1,3,5}{
	\node[v1] at (A.corner \i) {};
}
\foreach \i in {2,4,6}{
	\node[v0] at (A.corner \i) {};
}
\node[left=6] at (-1,-1) {\scriptsize $[1,4]$};
\node[left=12] at (-1,1) {\scriptsize $[3,4]$};
\node[right=-2] at (-1,0.5) {\scriptsize $[2,4]$};
\end{scope}
\node[below=30pt] at (3.8,0) {\small $[u,v] \leadsto (\Theta\merge_{\bf u}\Theta)\merge_{\bf v}\Theta$};
\end{tikzpicture}
\caption{All rooted bicubic planar maps on 6 vertices.}
\label{fig:library9edges}
\end{figure}

We finish with a much larger example, $n=14$, involving the Dyck path 
\[ \U^{18}\D^6\U^3\D^8\U^{18}\D^4\U^3\D^{24} \]
of semilength $42$ ($=3n$), with ascents decorated by $\Theta$ and two distinct rooted 6-prisms as shown in Figure~\ref{fig:21DyckPath}. Let $\widehat P$ denote the above Dyck path together with the rooted primitive maps decorating its ascents.

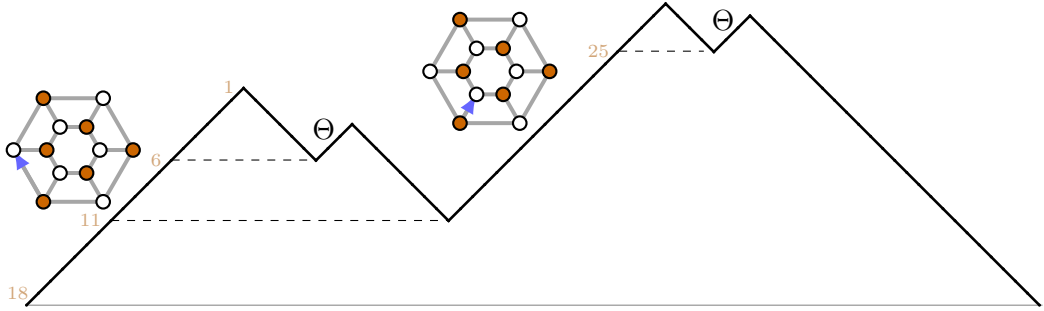
\begin{figure}[ht]
\begin{tikzpicture}[scale=0.16]
\dyckpath{0,0}{42}{1,1,1,1,1,1,1,1,1,1,1,1,1,1,1,1,1,1,0,0,0,0,0,0,1,1,1,0,0,0,0,0,0,0,0,1,1,1,1,1,1,1,1,1,1,1,1,1,1,1,1,1,1,0,0,0,0,1,1,1,0,0,0,0,0,0,0,0,0,0,0,0,0,0,0,0,0,0,0,0,0,0,0,0}
\node[label,left=0] at (18,18) {\tiny $1$};
\node[label,left=0] at (12,12) {\tiny $6$};
\node[label,left=0] at (7,7) {\tiny $11$};
\node[label,left=0] at (1,1) {\tiny $18$};
\node[label,left=0] at (49,21) {\tiny $25$};
\foreach \x/\y in {29/14.6,62/23.6} {\node[left=12pt] at (\x,\y) {$\Theta$};}
\draw[dashed] (12,12) -- (24,12);
\draw[dashed] (7,7) -- (35,7);
\draw[dashed] (49,21) -- (56,21);
\begin{scope}[xshift=110,yshift=365]
\node[regular polygon, regular polygon sides=6, minimum size=20, poly] at (0,0) (A) {};
\node[regular polygon, regular polygon sides=6, minimum size=45, poly] at (0,0) (B) {};
\foreach \i in {1,...,6} {\draw[poly] (A.corner \i) -- (B.corner \i);}
\draw[root] (B.corner 4) -- (B.corner 3);
\foreach \i in {1,3,5}{
	\node[v1] at (A.corner \i) {};
	\node[v0] at (B.corner \i) {};
}
\foreach \i in {2,4,6}{
	\node[v0] at (A.corner \i) {};
	\node[v1] at (B.corner \i) {};
}
\end{scope}
\begin{scope}[xshift=1090,yshift=550]
\node[regular polygon, regular polygon sides=6, minimum size=20, poly] at (0,0) (A) {};
\node[regular polygon, regular polygon sides=6, minimum size=45, poly] at (0,0) (B) {};
\foreach \i in {1,...,6} {\draw[poly] (A.corner \i) -- (B.corner \i);}
\draw[root] (B.corner 4) -- (A.corner 4);
\foreach \i in {1,3,5}{
	\node[v1] at (A.corner \i) {};
	\node[v0] at (B.corner \i) {};
}
\foreach \i in {2,4,6}{
	\node[v0] at (A.corner \i) {};
	\node[v1] at (B.corner \i) {};
}
\end{scope}
\end{tikzpicture}
\caption{Decorated Dyck path $\widehat P$.}
\label{fig:21DyckPath}
\end{figure}

Following the construction given in Section~\ref{sec:construction}, the associated rooted bicubic planar map is the one obtained by the merging operations encoded in $\widehat P$. Figure~\ref{fig:bigMerge} shows these operations with the corresponding labeled primitive components, and Figure~\ref{fig:bigMap} shows the assembled rooted bicubic map $M_{\widehat P}$ corresponding to the decorated Dyck path $\widehat P$.

\def\newmerge{\!+~\hskip-10pt+}

\begin{figure}[ht]
\begin{tikzpicture}
\begin{scope}
\node at (2,0) {$\newmerge_{\bf 6}$};
\node[regular polygon, regular polygon sides=6, minimum size=32, poly] at (0,0) (A) {};
\node[regular polygon, regular polygon sides=6, minimum size=78, poly] at (0,0) (B) {};
\foreach \i in {1,...,6} {\draw[poly] (A.corner \i) -- (B.corner \i);}
\draw[root] (B.corner 4) -- (B.corner 3);
\foreach \i in {1,3,5}{
	\node[v1] at (A.corner \i) {};
	\node[v0] at (B.corner \i) {};
}
\foreach \i in {2,4,6}{
	\node[v0] at (A.corner \i) {};
	\node[v1] at (B.corner \i) {};
}
\node[label,below=16, left=28] at (0,0) {\tiny $1$};
\node[label,below=4, left=22] at (0,0) {\tiny $2$};
\node[label,below=9, left=10] at (0,0) {\tiny $3$};
\node[label,below=20, left=11] at (0,0) {\tiny $4$};
\node[label,below=12] at (0,0) {\tiny $5$};
\node[label,below=26, right=5] at (0,0) {\tiny $6$};
\node[label,below=32] at (0,0) {\tiny $7$};
\node[label,below=20, right=26] at (0,0) {\tiny $8$};
\node[label,above=20, right=27] at (0,0) {\tiny $9$};
\node[label,above=32] at (0,0) {\tiny $10$};
\node[label,above=20, left=27] at (0,0) {\tiny $11$};
\node[label,above=22, left=10] at (0,0) {\tiny $12$};
\node[label,above=10, left=10] at (0,0) {\tiny $13$};
\node[label,above=12] at (0,0) {\tiny $14$};
\node[label,above=7, right=-2] at (0,0) {\tiny $15$};
\node[label,below=5, right=-2] at (0,0) {\tiny $16$};
\node[label,below=4, right=18] at (0,0) {\tiny $17$};
\node[label,above=22, right=10] at (0,0) {\tiny $18$};
\end{scope}

\begin{scope}[xshift=72,scale=0.85]
\node at (2.85,0) {$\newmerge_{\bf 11}$};
\draw[poly] (0,0) .. controls +(45:23pt) and +(135:23pt) .. (2,0);
\draw[root] (0,0) .. controls +(-45:23pt) and +(-135:23pt) .. (2,0);
\draw[poly] (0,0) -- (2,0);
\node[v1] at (0,0) {}; \node[v0] at (2,0) {};
\node[label,below=9pt] at (1,0) {\tiny $19$};
\node[label,above=9pt] at (1,0) {\tiny $20$};
\node[label,above=-2pt] at (1,0) {\tiny $21$};
\end{scope}

\begin{scope}[xshift=196]
\node at (2.1,0) {$\newmerge_{\bf 25}$};
\node[regular polygon, regular polygon sides=6, minimum size=32, poly] at (0,0) (A) {};
\node[regular polygon, regular polygon sides=6, minimum size=78, poly] at (0,0) (B) {};
\foreach \i in {1,...,6} {\draw[poly] (A.corner \i) -- (B.corner \i);}
\draw[root] (B.corner 4) -- (A.corner 4);
\foreach \i in {1,3,5}{
	\node[v1] at (A.corner \i) {};
	\node[v0] at (B.corner \i) {};
}
\foreach \i in {2,4,6}{
	\node[v0] at (A.corner \i) {};
	\node[v1] at (B.corner \i) {};
}
\node[label,below=20, left=11] at (0,0) {\tiny $22$};
\node[label,below=12] at (0,0) {\tiny $23$};
\node[label,below=26, right=1] at (0,0) {\tiny $24$};
\node[label,below=32] at (0,0) {\tiny $25$};
\node[label,below=20, right=27] at (0,0) {\tiny $26$};
\node[label,above=20, right=27] at (0,0) {\tiny $27$};
\node[label,above=32] at (0,0) {\tiny $28$};
\node[label,above=20, left=27] at (0,0) {\tiny $29$};
\node[label,below=16, left=28] at (0,0) {\tiny $30$};
\node[label,below=4, left=22] at (0,0) {\tiny $31$};
\node[label,below=9, left=10] at (0,0) {\tiny $32$};
\node[label,above=10, left=10] at (0,0) {\tiny $33$};
\node[label,above=12] at (0,0) {\tiny $34$};
\node[label,above=7, right=-2] at (0,0) {\tiny $35$};
\node[label,below=5, right=-2] at (0,0) {\tiny $36$};
\node[label,below=4, right=18] at (0,0) {\tiny $37$};
\node[label,above=22, right=11] at (0,0) {\tiny $38$};
\node[label,above=22, left=11] at (0,0) {\tiny $39$};
\end{scope}

\begin{scope}[xshift=196,yshift=100]
\node[rotate=-90] at (0,-1.7) {$\cong$};
\node[regular polygon, regular polygon sides=4, minimum size=44, poly] at (0,0) (A) {};
\node[regular polygon, regular polygon sides=4, minimum size=90, poly] at (0,0) (B) {};
\foreach \i in {1,...,4} {\draw[poly] (A.corner \i) -- (B.corner \i);}
\draw[root] (B.corner 2) -- (B.corner 3);
\foreach \i in {1,3}{
	\node[v0] at (A.corner \i) {};
	\node[v1] at (B.corner \i) {};
}
\foreach \i in {2,4}{
	\node[v1] at (A.corner \i) {};
	\node[v0] at (B.corner \i) {};
}
\draw[poly] (A.corner 4)+(-0.37,0) node[v0]{} -- +(-0.37,1.09) node[v1]{};
\draw[poly] (A.corner 4)+(-0.74,0) node[v1]{} -- +(-0.74,1.09) node[v0]{};
\node[label,left=30] at (0,0) {\tiny $22$};
\node[label,below=30] at (0,0) {\tiny $23$};
\node[label,right=30] at (0,0) {\tiny $24$};
\node[label,above=30] at (0,0) {\tiny $25$};
\node[label,above=19, right=18] at (0,0) {\tiny $26$};
\node[label,above=20,right=3] at (0,0) {\tiny $27$};
\node[label,above=14] at (0,0) {\tiny $28$};
\node[label,above=20,left=2.5] at (0,0) {\tiny $29$};
\node[label,above=19, left=18] at (0,0) {\tiny $30$};
\node[label,left=12] at (0,0) {\tiny $31$};
\node[label,below=19, left=18] at (0,0) {\tiny $32$};
\node[label,below=20,left=2.5] at (0,0) {\tiny $33$};
\node[label,below=14] at (0,0) {\tiny $34$};
\node[label,below=20,right=3] at (0,0) {\tiny $35$};
\node[label,below=19, right=18] at (0,0) {\tiny $36$};
\node[label,right=3] at (0,0) {\tiny $37$};
\node[label] at (0,0) {\tiny $38$};
\node[label,left=2.5] at (0,0) {\tiny $39$};
\end{scope}

\begin{scope}[xshift=272,scale=0.85]
\draw[poly] (0,0) .. controls +(45:23pt) and +(135:23pt) .. (2,0);
\draw[root] (0,0) .. controls +(-45:23pt) and +(-135:23pt) .. (2,0);
\draw[poly] (0,0) -- (2,0);
\node[v1] at (0,0) {}; \node[v0] at (2,0) {};
\node[label,below=9pt] at (1,0) {\tiny $40$};
\node[label,above=9pt] at (1,0) {\tiny $41$};
\node[label,above=-2pt] at (1,0) {\tiny $42$};
\end{scope}
\end{tikzpicture}
\caption{Merging operations encoded in $\widehat P$.}
\label{fig:bigMerge}
\end{figure}
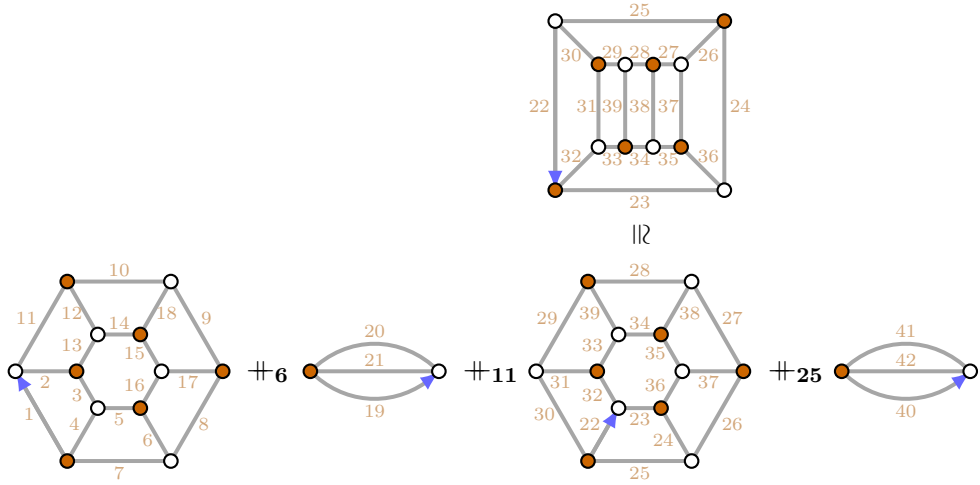

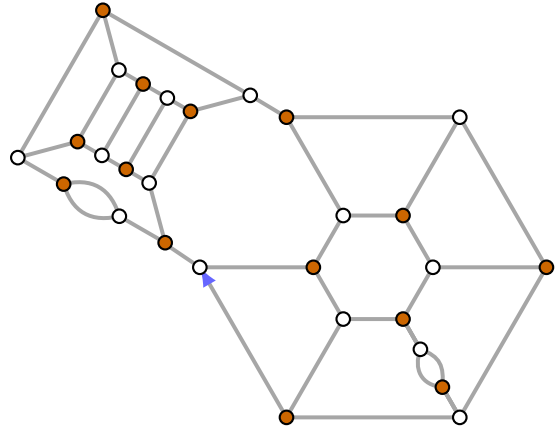
\begin{figure}
\begin{tikzpicture}
\begin{scope}
\node[regular polygon, regular polygon sides=6, minimum size=45, poly] at (0,0) (A) {};
\node[regular polygon, regular polygon sides=6, minimum size=130] at (0,0) (B) {};
\draw[poly] (B.corner 4) -- (B.corner 5) -- (B.corner 6) -- (B.corner 1) -- (B.corner 2) -- +(-0.4,0.24);
\draw[poly] (B.corner 3) -- +(-0.4,0.27);
\draw[poly] (A.corner 5) -- +(0.23,-0.4) .. controls +(-20:8pt) and +(85:8pt) .. +(0.52,-0.9) -- (B.corner 5);
\draw[poly] (A.corner 5) -- +(0.23,-0.4) .. controls +(250:8pt) and +(175:8pt) .. +(0.52,-0.9) -- (B.corner 5);
\draw[poly] (A.corner 5) -- +(0.23,-0.4) node[v0]{};
\draw[poly] (B.corner 5) -- +(-0.23,0.4) node[v1]{};
\foreach \i in {1,2,3,4,6} {\draw[poly] (A.corner \i) -- (B.corner \i);}
\draw[root] (B.corner 4) -- (B.corner 3);
\foreach \i in {1,3,5}{
	\node[v1] at (A.corner \i) {};
	\node[v0] at (B.corner \i) {};
}
\foreach \i in {2,4,6}{
	\node[v0] at (A.corner \i) {};
	\node[v1] at (B.corner \i) {};
}
\end{scope}
\begin{scope}[xshift=-90,yshift=53, rotate=150]
\node[regular polygon, regular polygon sides=4, minimum size=44, rotate=150, poly] at (0,0) (A) {};
\node[regular polygon, regular polygon sides=4, minimum size=90, rotate=150] at (0,0) (B) {};
\draw[poly] (B.corner 3) -- (B.corner 4) -- (B.corner 1); 
\draw[poly] (-0.45,1.1) .. controls +(45:12pt) and +(135:12pt) .. (0.45,1.1);
\draw[poly] (-0.45,1.15) .. controls +(-45:12pt) and +(-135:12pt) .. (0.45,1.15);
\draw[poly] (B.corner 1) -- +(-0.7,0) node[v1]{};
\draw[poly] (B.corner 2) -- +(0.7,0) node[v0]{};
\foreach \i in {1,...,4} {\draw[poly] (A.corner \i) -- (B.corner \i);}
\foreach \i in {1,3}{
	\node[v1] at (A.corner \i) {};
	\node[v0] at (B.corner \i) {};
}
\foreach \i in {2,4}{
	\node[v0] at (A.corner \i) {};
	\node[v1] at (B.corner \i) {};
}
\draw[poly] (A.corner 4)+(-0.37,0) node[v1]{} -- +(-0.37,1.09) node[v0]{};
\draw[poly] (A.corner 4)+(-0.74,0) node[v0]{} -- +(-0.74,1.09) node[v1]{};
\end{scope}
\end{tikzpicture}
\caption{Rooted bicubic planar map $M_{\widehat P}$ corresponding to $\widehat P$.}
\label{fig:bigMap}
\end{figure}


\end{document}